\documentclass[10pt]{amsart}
\usepackage{amsfonts}
\usepackage{amsmath,amssymb} 
\usepackage{enumerate} 
\usepackage{graphicx}
\usepackage{color}
\usepackage[hidelinks]{hyperref}
\usepackage{bm}
\usepackage{bbm}

\newtheorem{thm}{Theorem}[section]
\newtheorem{cor}[thm]{Corollary}
\newtheorem{lem}[thm]{Lemma}

\newtheorem{defn}[thm]{Definition}
\newtheorem{rem}[thm]{Remark}
\newtheorem{exm}[thm]{Example}

\numberwithin{equation}{section}
\newcommand{\norm}[1]{\left\Vert#1\right\Vert}
\newcommand{\abs}[1]{\left\vert#1\right\vert}
\newcommand{\mean}[1]{\mathbb{E}\lbrack #1\rbrack}

\newcommand{\set}[1]{\left\{#1\right\}}
\newcommand{\Real}{\mathbb R}

\newcommand{\bN}{\mathbb{N}}

\newcommand{\cA}{\mathcal{A}}

\newcommand{\cG}{\mathcal{G}}

\newcommand{\cN}{\mathcal{N}}

\allowdisplaybreaks
\begin{document}

\title[Convergence rate of DeepONets]{Convergence rate of {DeepONets} for learning operators arising from advection-diffusion equations}
\author[Deng et al.]{Beichuan Deng and Yeonjong Shin  and Lu Lu and Zhongqiang Zhang and George Em Karniadakis}
\address{Department of Mathematical Sciences, Worcester Polytechnic Institute, Worcester, MA. Email address: bdeng2@wpi.edu,\,\, zzhang7@wpi.edu.  }
\address{Department of Mathematics, Massachusett Institute of Technology, MA.  Email address: lu\_lu@mit.edu.  }
\address{Division of Applied Mathematics, Brown University, Providence, RI. 	 Email address: yeonjong\_shin@brown.edu \, george\_karniadakis@brown.edu.  }
\date{\today}

\begin{abstract}
We present convergence analysis of operator learning in [Chen and Chen 1995] and [Lu et al. 2020], 
where continuous operators are approximated by a sum of products of branch and trunk networks. 
In this work, we consider the rates of learning solution operators from both linear and nonlinear advection-diffusion equations with or without reaction.  We find that the convergence rates depend on the architecture of branch networks as well as the smoothness of inputs and outputs of solution operators.  %
\end{abstract}	
\maketitle

\section{Introduction}\label{intro} 
 Neural networks have been widely explored for solving differential equations, e.g., in  
 \cite{BerNys18,EYu18,RaiKar18,RaiPK19,KharazmiZK20} and many subsequent papers.  
 In these works,  solutions to differential equations are approximated by neural networks. Neural networks are thought of as alternatives to splines, orthogonal polynomials, or hp-finite elements bases.  One key advantage of neural networks is the capacity to approximate arbitrary continuous functions on compact domains. 
 However,  training neural networks are performed for differential equations with fixed inputs, such as initial conditions, boundary conditions, forcing, and coefficients. If one input is changed, the training process has to be repeated.  It is difficult to obtain outputs in real-time for  multiphysics systems  that require various sets.

 To overcome this difficulty, one can use operator learning, in which a fixed-weighted/pre-trained network approximates a continuous operator from the input(s) to the output(s), see, e.g., \cite{chen1995universal,li2020fourier,LuJinPange2021,SandbergXu1997}.  
In the seminal work  \cite{chen1995universal}, the authors approximate a continuous nonlinear operator by a network,  which is a summation of products of two two-layer networks with fixed weights. The idea is to approximate the basis expansions of operators in separable Banach spaces. Let $\cG(\cdot)$ be the operator of interest, which is represented using a  Schauder basis $\set{e_k}$ in the Banach space by 
$\cG(u)(y)= \sum_{k} c_k(\cG(u)) e_k(y)$, where $ c_k(\cdot)$ is a linear functional. 
Then,  the Schauder basis $e_k$ is approximated by a two-layer network (called trunk network in \cite{LuJinPange2021}). 
The   functional $ c_k(\cG(u))$ can   be  approximated by 
a two-layer network (called a branch network in \cite{LuJinPange2021}) as   $ c_k(\cG(u))$ can be first approximated by a continuous function $ c_k(\cG(\ell(\mathbf{u}_m)))$ with $\ell (\mathbf{u}_m)$ being an approximation of $u$.  
In \cite{BacKChen02}, the idea of operator learning has been extended to parameterized multiple operators. 
In \cite{LuJinPange2021}, the two-layer networks are replaced by multi-layer networks (see also Theorem \ref{thm:universal-apprx-operator} below), where the networks are named DeepONets. In \cite{LuJinPange2021}, many numerical examples of learning both explicit and implicit operators are presented,  demonstrating the efficiency of DeepONets. The implicit operators include  solution operators from nonlinear ordinary differential equations and advection-diffusion-reaction equations.

A similar operator learning approach was developed for dynamical systems in \cite{SandbergXu1997a,SandbergXu1997}. 
In  \cite{li2020fourier}, 
a different way to construct the networks is proposed for solution operators from partial differential equations. Instead of summations of products of neural networks, feedforward multi-layer networks are used where nonlinear kernels are applied in layers to accommodate infinite-dimensional inputs. 
In \cite{guss2019universal}, some theoretical results have been presented on how to use convolutional neural networks for operator learning, but no computational experiments were performed.

A neural network operator learning method  includes two steps:  
a)  solve the given differential equation numerically with various inputs or collect observation data  and train the network (offline),
b) update the training results with the corresponding new inputs without re-training (online).
Since the network resulting from a) has fixed weights, Step b) can be done efficiently as only evaluations are needed.  
However, the cost of Step a) is high. It is unclear how the cost depends on the training data, e.g., the number of inputs and the number of points/parameters needed to represent one input.

In this work, we focus on the dependence on the number of points/parameters needed to represent one input (approximation theory only) while we do not consider training/optimization error for neuron networks.  Specifically, we discuss
 the convergence in the number of branch and trunk networks and the sizes of each branch and trunk network  in DeepONets, and the parameterization of the input $u$, see Theorem \ref{thm:universal-apprx-operator} for the networks in DeepONets. 
The main difficulty of the analysis for  \cite{chen1995universal,LuJinPange2021} is that the approximation $ c_k(\cG (\ell(\mathbf{u}_m)))$  is a high dimensional function ($m$ is large) in general. Although there are a few theoretical results for approximating high-dimensional continuous functions, they are not sufficient to show the operator learning's superiority.    If such functions  are only Lipschitz continuous, one may encounter the curse of dimensionality as the convergence rate is of the form 
$(\cdot)^{-1/m}$, 
 see, e.g., \cite{mhaskar1997neural} and  Theorem \ref{thm:deep-o-nets-error} in this work.  If such functions are analytic, we may break the curse of dimensionality; see, e.g., \cite{OpsSZ2020} for high dimensional functions. Unfortunately,  these functions (functionals) are not analytic in most cases, and only limited smoothness can be assumed, see, e.g., \cite{chen1993approximations,Chen1998,mhaskar1997neural,Sandberg1992,Sandberg96}.

Instead of investigating the smoothness only, we observe  that 
many solution operators from differential equations admit special structures such that no so-called curse of dimensionality (concerning the nominal dimension $m$) can occur. For example,  we show in Section \ref{chap4_burgers_advection} that one may use rational functions and ReLU  neural networks to approximate solution operators from Burgers and linear advection-diffusion equations. Consequently,   realistic convergence rates are obtained without curse of dimensionality for $\mathbf{u}_m$ (nominal dimension $m$), see e.g., in Theorems \ref{thm:burgersoprtest1} and \ref{thm:1dadvecnetworkerr}.  We also design    branch networks with  special structures for solution operators from linear advection-diffusion-reaction equations. The designed branch networks admit ``blessed presentations" \cite{mhaskar2016deep}. The key idea is to use a finite difference scheme to approximate  a solution and utilize an appropriate iterative solver to represent the numerical solution  explicitly; see Section \ref{sec:2d-reaction-diffusion}. A different approach of proving the convergence rates of DeepONets for analytic operators is presented in \cite{lanthaler2021error} using reproducing Hilbert spaces.

 The rest of the paper is organized as follows.  In Section \ref{sec:thm:operator}, we present an universal approximation theorem for operators and discuss the ideas of deriving  convergence rates. We also present the main results of the work in this section and  present the proofs in Sections \ref{chap4_burgers_advection} and  \ref{chap5_proofs}. 
 In Section \ref{sec:error-deeponets-bound-generic}, we present a  generic convergence rate where the operator is assumed to be H\"older continuous. 
 In Section \ref{chap4_burgers_advection}, we present the proofs for Burgers and steady-state linear advection-diffusion equations. We show that the branch networks are approximated by deep ReLU (rectified linear unit) neural networks via rational polynomials.  In the appendix, we present some key technical lemmas and theorems needed in our proofs. 



\section{Main results} \label{sec:thm:operator}

 In this work, a convergence rate is described by the necessary capacity of the neural network to achieve the given accuracy $\varepsilon$. Let $\mathcal{N}(x;\theta)$ be a neural network, where $x$ denotes the input(s) and $\theta$ denotes the parameters (weights and biases). We will use the following notations to describe the network capacity unless otherwise specified: \\
a) The size of $\mathcal{N}$ is the total number of nonzero parameters (weights and biases), denoted by $| \theta |$. \\
b) The width of $\mathcal{N}$ is the number of neurons in each layer, denoted by $N_{\mathcal{N}}$.\\
c) The depth of $\mathcal{N}$ is the number of layers, denoted by $L_{\mathcal{N}}$.

In this section, we present the universal approximation theorem from   \cite{LuJinPange2021} and discuss issues in proving convergence rates for DeepONets. We also present the  main results of this work.

\begin{thm}[\cite{LuJinPange2021}] \label{thm:universal-apprx-operator}
Let $K_1 \subset \Real^{d_1}$ be a compact set. Let $V$ be a compact set in $C(K_1)$, $K_2 \subset \mathbb{R}^d$ be a compact set and $Y = C(K_2)$. 
	Assume that $\cG:\, V\to Y$ is a nonlinear continuous operator.  
	%
	 Then for any $\epsilon>0$, there are positive integers  $p$ and $m$,  neural networks $f^{\mathcal{N}}(\cdot;\theta^{(k)})$ and $g^{\mathcal{N}}(\cdot;\Theta^{(k)})$, $x_j \in K_1$,  $k=1,\dots,p$, $j=1,\dots,m$, such that
	\begin{equation}\label{branchtrunkformula}
	\left|\cG(u)(y) - \cG_{\bN} (\mathbf{u}_m)
	\right|<\epsilon , \quad \cG_{\bN} (\mathbf{u}_m)=\sum_{k=1}^p
 	\underbrace{ g^{\mathcal{N}}(\mathbf{u}_m;\Theta^{(k )})}_{branch}
	\underbrace{f^{\mathcal{N}}(  y;\theta^{(k)})}_{trunk}
	\end{equation}
	holds for all $u \in V$ and $y \in K_2$, where $\mathbf{u}_m=(u(x_1), u(x_2),\cdots, u(x_m))^\top$. The neural networks $f^{\mathcal{N}}$ and $g^{\mathcal{N}}$ can be any class of functions that satisfy a classical universal approximation theorem  of continuous functions  on compact sets.  
\end{thm}

\begin{rem}
	Here $C(K_1)$ can be replaced by  $L^p(K_1)$, $p\geq 1$ and $C(K_2)$ can be replaced by 
	$L^q(K_2)$, $q\geq 1$. But 
	$\mathbf{u}_m$ should be replaced by averaged values 
	$\mathbf{u}_m= ( \mathbf{u}_h (x_1), \ldots,\mathbf{u}_h (x_m))^\top$, where 
$	\mathbf{u}_h (x_i)  = \int_{B(x_i,h)\cap K} u(t)\,dt/{\mu (B(x_i,h)\cap K)} $
	while $\mu$ is the Lebesgue measure and $B(x_i,h)$ is the ball centered at $x_i$ with radius $h$. More generally, these spaces can be replaced by 
	separable Banach spaces.  
\end{rem}

When the networks $f^{\cN}$ and $g^{\cN}$ have two layers, the theorem is proved in \cite{chen1995universal}. 
If more layers are involved,  the proof is based on the universal approximation theorem for continuous functions on compact sets
and the fact that two-layer networks in \cite{chen1995universal} are continuous on compact sets, see \cite{LuJinPange2021}.

%
Theorem \ref{thm:universal-apprx-operator} covers neural networks including feed-forward neural networks with non-polynomial activation functions (e.g.,  \cite{hornik1989multilayer}), radial basis network \cite{chen1995approximation}, convolutional neural networks, ResNets and more. The only requirement is that the networks admit the universal approximation of continuous functions on compact sets. 

\begin{rem}
It is possible to extend Theorem \ref{thm:universal-apprx-operator} to  discontinuous operators
 	that can be approximated by continuous operators,
	e.g., in 
	\cite{SandbergXu98} for functionals.  The condition of compact set in $\Real^d$ may be also replaced by 
	$\Real^d$ if the input/output function decays fast enough at infinity.
\end{rem}

\subsection{Framework of analysis}
First of all, $\cG : V \to Y$ is approximated by an operator 
$\cG_{m,p}: V_m \to Y_p$, where $V_m\subset V$ and $Y_p \subset Y$ are both finite-dimensional subspaces.   Assume that the domains of $K_1$ and $K_2$ are cubes or  balls, where  approximations with convergence rates are well studied. If they are not cubes or balls, we use the Tietze-Urysohn-Brouwer extension theorem to 
  extend the operator such that  domains of input and output are cubes or balls (denoted by $D$).   We denote the resulting operator (with also extension of $V$ and image) by $\cG(u)$ if no confusion arises. 
Second, we apply universal approximation theorem for continuous function on compact sets -- there exist neural networks  that can approximate the possibly high-dimensional  function
$\cG_{m,p}(\mathbf{u}_m)$ sufficiently well -- to obtain the existence of neural network and analyze the convergence rate with respect to  the size of branch networks further.

\begin{rem} 
	How to choose the spaces $V$, $V_m$, $Y$ and $Y_p$ is essential to determine the convergence rate.
	Here are some typical choices of these spaces. 
	\begin{table}[!h]
		\centering
		\caption{Here  RKHS refers to reproducing Kernel Hilbert space. }
		\begin{tabular}{l|l|l  }  \hline 
			$V$ or $Y$   &  $V_m$ or $Y_p$ &       Ref.  \\ \hline
			$L^p $ &  piecewise polynomials	   & this work   \\ \hline
			$L^q $ & truncated  Fourier space  &     \cite{li2020fourier}	         \\\hline
			RKHS  & truncated RKHS (first $N$ term of the basis) &  \cite{lanthaler2021error}  \\ \hline		
		\end{tabular}

	\end{table}
	It is important to choose the best possible choices to parameterize the input and output as they determine the dimensionality of the  input and output of the operator $\cG_{m,p}(\mathbf{u}_m)$. 	However,  we only consider two choices  of piecewise polynomials and truncated Fourier space in this work. 
\end{rem}

Below is one way to obtain such $\cG_{m,p}$, using   piecewise constant interpolation and orthonormal expansion. 
\begin{eqnarray}\label{erroridea}
\cG(u )
&  \approx &  \cG(\mathcal{I}_{m,x}^0u)   \quad\text{(piecewise constant interpolation)}\nonumber\\ 
&\approx & \sum_{k=1}^p   \int_{D}  	 \cG(\mathcal{I}_{m,x}^0u) \, e_k(y)\,dy\, e_k(y)  ,\quad (K_2\subset D, \, \text{spectral expansion}) \nonumber\\
&\approx&    \sum_{k=1}^p 
\int_{D}  \mathcal{I}_{n,y} \Big(	\cG(\mathcal{I}_{m,x}^0u)(y) \big)  e_k(y)\,dy \,  e_k(y)\quad  \text{ (interpolation})  \nonumber\\
&= &   \sum_{k=1}^p \Big(
\sum_{i=1}^n   	\cG(\mathcal{I}_{m,x}^0u)(y_i)   \underbrace{\int_{D}  e_k(y)\phi_i(y)\,dy}_{c_{i}^k} \Big) \, e_k(y) =: \cG_{m,p} (\mathbf{u}_{m}).
\label{gmpdef}
\end{eqnarray}  
Here, $\mathbf{u}_{m} = (u(x_1), \cdots, u(x_m))^\top$, and $y_i$, $\phi_i$ are quadrature points and basis functions with respect to $ \mathcal{I}_{n,y}$.   
Observe that $	\cG(\mathcal{I}_{m,x}^0u)(y_i)$ is  continuous with respect to $\mathbf{u}_{m}$, on  $[-M,M]^{m}$,  $M= \max_{ 1\leq  i\leq m} \sup_{u\in V}\abs{u(x_i)}$. Then  by the universal approximation theorem, 
we have the following  uniform approximation     
\begin{eqnarray}\label{brancherror}
\sup_{u\in V}\sup_{\mathbf{u}_m \in [-M,M]^{m}}\abs{	\cG(\mathcal{I}_{m,x}^0u)(y_i) -g^{\mathcal{N}}( \mathbf{u}_m;\Theta_{y_i})  }<\epsilon.
\end{eqnarray}
Since the  basis $e_k(y) $ can be well approximated by a neural network $ f^{\mathcal{N}}({ y}; \theta^{(k)})$, the convergence can be established. Another way to obtain $\cG_{m,p} (\mathbf{u}_{m})$
is using interpolations only:  
$\cG(u )
   \approx 	\sum_{k=1}^p\cG(\mathcal{I}_{m,x}^0u)(y_k) \phi_k(y) $. 


When the operator $\cG$ is H\"older continuous, we show the convergence rate of the 
DeepONets in Theorem \ref{thm:deep-o-nets-error}.  While for  analytic operators, exponential convergence can be obtained  \cite{lanthaler2021error}. We remark that 
  the H\"older continuity leads to slow convergence rates.  Moreover,  many solution operators may not be analytic. In addition to the smoothness of the operators, the structure of the approximate or analytic solution is also important to the convergence rates of DeepONets for these operators.  For solution operators from Burgers equations and   steady linear advection-reaction-diffusion equations, we show better convergence rates below.
  We will observe that the rate of convergence will heavily depend on the nature of the problem and the formulation to obtain  $\cG(\mathcal{I}_{m,x}^0u)(y_i)$, analytically or numerically.

 In what follows, we consider that the input is piecewise constant or linear. By the assumption of H\"older continuity \eqref{eq:operator-holder-continuity},  
$\norm{\cG(v) -\cG(\mathcal{I}_{m,x}^0 v)}_Y\leq C \norm{v-I_{m,x}^0 v}_X^\alpha$, where 
$\alpha=1$ holds for examples  below  with corresponding $X$ and $Y$ in these  examples.  
\subsection{1D Burgers equation}

Consider the 1-D Burgers equation with  periodic boundary condition 
\begin{eqnarray}
\left\{\begin{array}{ll}
u_t + u  u_x = \kappa u_{xx}, \ \ (x,t) \in \Real \times(0,\infty),\ \ \kappa>0,\\ 
u(x-\pi,t) = u(x+\pi,t ), \quad
u(x,0)=u_0(x).
\end{array}\right.
\label{eqtburgers}
\end{eqnarray}
Let $M_0,\,M_1>0$. Define
\begin{eqnarray}
 \mathcal{S} = \mathcal{S}(M_0, M_1) := \{ v\in W^{1,\infty}(-\pi,\pi): \norm{v}_{L^{\infty}} \leq M_0, \| \partial_x v\|_{L^{\infty}} \leq M_1, \bar{v}:=\int_{-\pi}^{\pi}v(s)\,ds = 0 \}.
\end{eqnarray}
We consider $u_0 \in \mathcal{S}$ to be a piecewise linear function, i.e., 
$u(x,0)=u_0(x) = \sum_{j=0} ^{m-1} u_{0,j} L_j(x)$,
where  
$-\pi = x_0 < x_1<\cdots <x_m = \pi$, $u_{0,j} = u_0(x_j)$, and $L_j(x)$ is the piecewise linear nodal basis supported on the sub-interval $[x_{j-1}, x_{j+1})$, $h_j=  x_j - x_{j-1} $ and $\max_{1\leq j \leq m} h_j  \leq h$. Let $\mathbf{u}_{0,m} = (u_{0,0}, u_{0,1},\cdots, u_{0,{m-1}})^\top$.

Here is the main result on the convergence rate of DeepONets approximating the solution operator of the 1-D Burgers equation with periodic boundary condition \eqref{eqtburgers}.

\begin{thm}[\textbf{1D Burgers equation with periodic boundary}]\label{thm:main_burgers}
Let $u_0\in \mathcal{S}$ be a piecewise linear function. Let $\mathcal{G}(u_0)$ be the solution operator of the Burgers equation 
\eqref{eqtburgers}. Then, there exist  ReLU branch networks $g^{\mathcal{N}}(\mathbf{u}_{0,m}; |\Theta^{(i)}|)$of size $| \Theta^{(i)} | = \mathcal{O}(m^2 \ln(m))$ for $i=1,\cdots, p$, 
and  ReLU trunk networks $f^{\mathcal{N}}( x ;\theta^{(k)})$ having width $N_{f_{\cN}}=3$ and depth $L_{f_{\cN}}=1$, $k=1,\cdots, p$, such that 
\begin{eqnarray*}
\norm{\cG (u_0 )-\cG_{\bN}(\mathbf{u}_{0,m} ) }_{L^{\infty}} \leq C \Big( p^{-1} + m^{-1} +\abs{\Theta^{(i)}}^{-\frac{1}{2}+\epsilon}  \Big),
\end{eqnarray*}
where   $\cG_{\mathbb{N}}(\mathbf{u}_{0,m})$  is  of the form in \eqref{branchtrunkformula}, $\epsilon>0$ is arbitrarily small and $C >0$ is independent of $m$, $p$, $|\Theta^{(i)}|$ and the initial condition $u_0$.  
\end{thm}

\begin{rem}
If the    initial condition $u_0(x)$ has the average $\bar{u}_0\neq 0$,  we write the solution  
	$u(x,t):= v(x-\bar{u}_0 t, t) +\bar{u}_0$, where $v$  satisfies the Burgers equation \eqref{eqtburgers} with the initial condition  of zero average  $u_0-\bar{u}_0$.  
\end{rem}

If the Burgers equation is equipped with Dirichlet boundary condition, or with a forcing term, or in two dimensions, similar conclusions still hold.  See Section \ref{chap4_burgers_advection} for details.

\subsection{1D advection-diffusion equations}
Consider the following 1D advection-diffusion equation with Dirichlet boundary condition:
\begin{eqnarray}\label{1dadvection}
\left\{\begin{array}{ll}
-u_{xx} + a(x)u_x = f(x), \ \ x \in (0,L), \,0<L<\infty,\\
u(0) = u(L) = 0. \\
\end{array}\right.
\label{eqtadv}
\end{eqnarray}
where $a(x)$, $f(x) \in L^{\infty}(0,L)$. Since $u(x)$ depends linearly on $f(x)$,   we are mainly concerned that how $u(x)$ depends on $a(x)$. Hence, we set $a(x)$ as the only input to the operator  only and fix $f(x)$. Let $M_0>0$. Define
\begin{eqnarray}
 {\Sigma} ={\Sigma}(M_0) := \{ a(x) \in L^{\infty}(0,L) : \| a\|_{L^{\infty}} \leq M_0 \}.
\end{eqnarray}
  Suppose that  $a \in \mathcal{S}$ is piecewise constant, i.e.,  
$	a(x) = \sum_{j=1}^m a_j  \, \chi_j(x) $, 
where $0 = x_0 < x_1<\cdots <x_m = L$, and $\chi_i(x)$ is the characteristic function on the sub-interval $(x_{j-1}, x_j)$, i.e., $\chi_i(x)=1$ if $x\in (x_{j-1}, x_j)$ and is otherwise $0$. 
Let $\mathbf{a}_m = (a_1,\cdots,a_m)^\top$.

Here is the main result on convergence rate of DeepONets for the problem \eqref{1dadvection}, whose proof can be obtained by using Lemma \ref{prop:1dadvecrationalerr} and following the idea of the proof of Theorem \ref{thm:main_burgers}. 
\begin{thm}[\textbf{1D advection-diffusion with Dirichlet boundary condition}]\label{thm:main_1dadvec}
	 Let $a \in \Sigma$ be piecewise constant. 
	For any given $f\in L^{\infty}$, let $\mathcal{G}^f(a)$ be the solution operator given by \eqref{1dellipticexcslt} .  Then, there exist ReLU branch networks $g^{\mathcal{N}}_f (\mathbf{a}_m; \Theta^{(i)})$  of size $|\Theta^{(i)}|= m^4 \ln(m)$ for $i=1,\cdots, p$,
and  ReLU trunk networks $f^{\mathcal{N}}( x ;\theta^{(k)})$ of width $N_{f_{\cN}}=3$ and depth $L_{f_{\cN}}=1$, $k=1,\cdots, p$, such that 
\begin{eqnarray*}
\norm{\cG^f (a )-\cG_{\bN}^f(\mathbf{a}_{m}) }_{L^{\infty}} \leq C\Big( p^{-1} + m^{-1} + \abs{\Theta^{(i)}}^{-\frac{1}{4}+\epsilon} \Big),
\end{eqnarray*}
where $\cG_{\bN}^f(\mathbf{a}_{m})$ is  of the form in  \eqref{branchtrunkformula}, $\epsilon>0$ is arbitrarily small and $C >0$ is independent of $m$, $p$, $|\Theta^{(i)}|$ and $a(x)$.
\end{thm}

\subsection{2D advection-reaction-diffusion equations}
Consider the following 2-D advection-reaction-diffusion equation with a given boundary condition 
\begin{eqnarray}\label{2dadvectionreactionpde}
\left\{\begin{array}{ll}
-\Delta u + \mathbf{a}\cdot  \nabla u + a_3(x,y) u = f, \quad \hbox{in} \ \Omega\subset \Real^d,\\
\mathcal{B}u = 0, \quad \hbox{on} \ \partial \Omega,
\end{array}\right.
\end{eqnarray}
where $ \mathbf{a} = [a_1(x,y), a_2(x,y)]$, $\Omega$ is a rectangular domain and $\mathcal{B}$ is one of the Dirichlet, Neumann or Robin boundary operator. Herein, we consider a finite difference discretization in order to obtain a numerical solution $U_N\in \Real^{m}$. Suppose that $S$ is the matrix resulting from a central finite difference scheme for $-\Delta$ and $D_x$ for partial derivative $a_1\partial_x$ and 
$D_y$ for partial derivative $a_2\partial_y$ and $\Lambda$ for $a_3 \text{Id}$  (Id refers to the identity operator). We then write the matrix  as follows
\begin{eqnarray*}
\Big( S+ D_x + D_y + \Lambda \Big) U_m := \Big( S+ h \sum_{i=1}^m a_1^i A_1^i + h \sum_{j=1}^m a_2^j A_2^j +\sum_{l=1}^m h^2 a_3^l A_3^l \Big) U_N = F,
\end{eqnarray*}
where $a^k_i$ denotes the value of $a_i(x,y)$, $i=1,2,3$, at the $k$-th node of 
$\{ (x_{j_1}, y_{j_2})\}_{j_1,j_2=1}^{\sqrt{m}}$ associated with the finite difference scheme, and every $A_i^k$ is a rank-1 matrix, $i=1,2,3$ and $k=1,\cdots,m$. Let $\mathbf{a}_1^m = (a_1^1,\cdots,a_1^m)^\top$, $\mathbf{a}_2^m = (a_2^1,\cdots,a_2^m)^\top$ and $\mathbf{a}_3^m = (a_3^1,\cdots,a_3^m)^\top$. The  theorem below states the convergence rate of DeepONets approximating the solution operator of 2D advection-reaction-diffusion equation \eqref{2dadvectionreactionpde}. The detailed discussion is presented in Section \ref{chap4_burgers_advection} for $\mathbf{a}=0$ or ${a}_3=0$.  

\begin{thm}[\textbf{2D advection-diffusion-reaction with classical boundary conditions}]\label{thm:main_2dadvreac}
	For any given $f\in L^{\infty}$, let $\cG^f(a_1,a_2,a_3)$ be the solution operator of Equation 
	\eqref{2dadvectionreactionpde}.  Then, there exist a branch network $g^{\mathcal{N}}  (\mathbf{a}_1^m, \mathbf{a}_2^m, \mathbf{a}_3^m; \Theta)$  having width $N_{g^{\cN}}= \mathcal{O}( m^2 \ln(m))$ and depth $L_{g^{\cN}} = \mathcal{O}( m\ln(m))$,
and  ReLU trunk networks $f^{\mathcal{N}}( x ;\theta^{(k)})$ having width $N_{f_{\cN}}=\mathcal{O}(1)$ and depth $L_{f_{\cN}}=\mathcal{O}(1)$, $k=1,\cdots, p$ (we take $p=m$), such that
\begin{eqnarray*}
\norm{\cG^f(a_1,a_2,a_3)-\cG_{\mathbb{N}}^F (\mathbf{a}_1^m, \mathbf{a}_2^m, \mathbf{a}_3^m) }_{L^{\infty}} \leq C \Big( m^{-1+\epsilon} + m^{-\frac{r}{2}} + \abs{N_{g^{\cN}} L_{g^{\cN}}}^{-\frac{1}{3}+\epsilon} \Big),
\end{eqnarray*}
where 
  $\cG_N^F(\mathbf{a}_1^m, \mathbf{a}_2^m, \mathbf{a}_3^m)$ is  the DeepONets of the form  in  \eqref{branchtrunkformula}, $\epsilon>0$ is arbitrarily small and $C>0$ is independent of $m$, $N_{g^{\cN}}$, $L_{g^{\cN}}$, $N_{f^{\cN}}$, $L_{f^{\cN}}$, $a_1$, $a_2$ and $a_3$. Here $r$ is the   convergence order of central finite difference schemes for  Equation \eqref{2dadvectionreactionpde} and $s$ is the regularity index of $u$ with given $f$ and $a_i$, $i=1,2,3$.
\end{thm}

\begin{rem}\label{rmk:dimdependence}
The error estimates of DeepONets depend on the dimension of the physical space of the PDE, i.e., the dimension $d$ where $\Omega \subset \mathbb{R}^d$. If the regularity of the solution $u$ is limited and piecewise constant/linear interpolation is used, then  the error of the deepONets is bounded by 
$ \Big(p^{-\frac{s}{d}} + m^{-\frac{r}{d}} + \hbox{error from branch}\Big)$, 
where $s$ is the regularity index of the analytical solution $u$.

\end{rem}

	 The branch network we construct from a  finite difference solution has the size $\mathcal{O}(m^3 \ln(m))$ in Theorem \ref{thm:main_2dadvreac} while the branch network  we derived from the analytical solution in Theorem \ref{thm:main_1dadvec}  has the size $\mathcal{O}(m^4\ln(m))$.
This is due to the different structures of the networks. In  Theorem  \ref{thm:main_1dadvec}, we use a generic ReLU network while in Theorem \ref{thm:main_2dadvreac} we use a blessed representation as in \cite{mhaskar2016deep}. The blessed representation has a structure similar to an iterative solver of the linear system. See Section \ref{chap4_burgers_advection} for more details. 

\begin{rem}[General domains]
	In this work, we consider regular domains such as periodic domains or rectangular domains.  The methodology developed for the advection-diffusion-reaction equations on rectangles can be extended to 
	irregular domains as the properties we use in the proof are sparsity of the finite difference matrices,   the iterative solver based on Sherman-Morrison's formula, and the matrix $S$ is invertible; see details in Section \ref{chap4_burgers_advection}. 
	
\end{rem}

\section{Error of operator approximation: general case}\label{sec:error-deeponets-bound-generic}

In this section, we present a general framework of convergence analysis for approximating H\"older continuous operators by DeepONets. Below is the definition of a H\"older continuous operator.


\begin{defn}
	Let $X$ and $Y$ be Banach spaces. 
	The operator $\cG:\,X\to Y$ is called H\"older continuous on a bounded subset $\mathcal{S}\subset X$, if
	\begin{equation}\label{eq:operator-holder-continuity}
	\norm{\cG (u )-\cG (v) }_{Y} \leq C  \norm{u-v}_{X}^{\alpha}, \quad \forall u,v \in \mathcal{S}, \quad 0<\alpha \leq 1,
	\end{equation}
	where $C>0$ depends  only on  $\mathcal{S}$ and the norms of $u$ and $v$  and the operator $\cG$. 
\end{defn}

There are many classes of operators, which are H\"older or Lipschitz continuous, especially solution operators from integral and differential equations. Various operators have been shown to  
satisfy H\"older continuity in the Appendix of \cite{LuJinPange2021}. For readers' convenience, we present an example below. 
\begin{exm}[\cite{holden2015front}]
	Consider the following equation 
	\begin{equation*}\partial_t u + {\rm div}f (u ) = 0,  \quad x\in \Real^d, \quad u(x,0)=u_0(x).\qquad\qquad \qquad\qquad 
	\end{equation*}
	The  solution operator $u=\cG(u_0,f)$ is Lipschitz continuous. That is, if
$\partial_t v + {\rm div}g (v ) = 0$, $x\in \Real^d$, and  \quad $v(0,x)=v_0(x)$, then 
	\[ \norm{u(t)-v(t)}_{L^1}\leq \norm{u_0-v_0}_{L^1} + t \min\set{{\rm TV}(u_0),{\rm TV}(v_0)}\norm{f-g}_{C^1},\quad t>0. \]
Here $u_0,v_0\in \mathcal{S} = \{ u\in {\rm BV}(\Real^d)\cap L^1 (\Real^d): {\rm TV}(u) \leq M\}$ for some $M>0$, $f$ and $g\in C^1(\Real,\Real^d)$ which  satisfies the entropy condition, where $\rm{BV}$	 is the space of functions of bounded variation and $\rm{TV}$ is the total variation. 
 
\end{exm}

In order to analyze the convergence rate, we split the error into two parts in the following way
\begin{equation}
\norm{\cG(u) - \cG_{\bN}(\mathbf{u}_{m})}_Y\leq  \norm{\cG(u) - \cG_{m,p}(\mathbf{u}_{m})}_Y+
\norm{\cG_{m,p}(\mathbf{u}_{m}) - \cG_{\bN}( \mathbf{u}_m)}_Y.
\end{equation}
where $\cG_{m,p}(\mathbf{u}_{m})$ is  a finite dimensional operator from $V_m$ to 
$Y_p$,   see, e.g, \eqref{gmpdef}. 
The first step is to find an appropriate operator $\cG_{m,p}$. 
One choice is the Bochner-Riesz means of Fourier series of a function $f(x)$, denoted by 
\begin{eqnarray*}
B_{R}^\gamma(f;y) = \sum_{\abs{k}\leq R} (1-\frac{\abs{k}^2}{R^2})^{\gamma} \widehat{f}_k  \exp(i \pi k^\top x ):=  \sum_{i=1}^p c_i(f)   e_i(y), \quad R>0 \quad \hbox{and} \quad \gamma\geq  0,
\end{eqnarray*}
where  $k= (k_1,k_2,\ldots,k_d)$, $\abs{k}^2= \sum_{i=1}^d \abs{k_i}^2$, $\widehat{f}_k$'s are the Fourier coefficients of $f$, and $c_i(f)$ can be considered as linear functionals of $f$.  The truncation error of $B_{R}^\gamma$ may be described by the first-order and second-order modulus of  continuity of $f$ in $L^q$; see definitions below.
\begin{eqnarray*}
	&\omega_1(f;t)_q = \sup_{\abs{h}\leq  t}\norm{f(\cdot+h)-f(\cdot)}_{L^q},& \quad t\geq 0,\\
	&\omega_2(f;t)_q = \sup_{\abs{h}\leq  t}\norm{f(\cdot+h)+f(\cdot -h)-2f(\cdot)}_{L^q},&\quad t\geq 0. 
\end{eqnarray*}
Then we can obtain the formulation of $ \cG_{m,p} (\mathbf{u}_{m})$ according to the calculation in \eqref{gmpdef}:
\begin{equation}\label{eq:generic-finite-d-G-approximate}
\cG_{m,p}(\mathbf{u}_m) =  B_R^\gamma \Big(\cG (\mathcal{I}_m^0 u);y \Big),
\end{equation}
where $\mathcal{I}_m^0$ is the piecewise constant interpolation that satisfies ($C$ is a generic constant and $h \sim m^{-\frac{1}{d}}$ ):
\begin{eqnarray}\label{interpltcdt}
&\norm{\mathcal{I}_m u-u}_X\leq C h\norm{u}_V, \forall u\in V.
\end{eqnarray}
Note that \eqref{interpltcdt} implies the stability of interpolation, i.e., $\norm{\mathcal{I}_m u}_X\lesssim  \norm{u}_V, \forall u\in V$. If we further consider the neural network $\cG_{\mathbb{N}}(\mathbf{u}_m)$ to be \eqref{branchtrunkformula}, then we have the following error estimate. 
 

\begin{thm}\label{thm:deep-o-nets-error}
	Assume that the conditions in Theorem \ref{thm:universal-apprx-operator},
	\eqref{eq:operator-holder-continuity} and \eqref{interpltcdt} hold.  Let $Y=L^q(K_2)$, where $1\leq q \leq \infty$, $K_2 \in \mathbb{R}^d$ is compact. Then  
	\begin{eqnarray}\label{ineq:universal-apprx-operator}
	\norm{\cG (u )-\cG_{\bN}(\mathbf{u}_m ) }_{Y} &\leq& 
	\norm{\cG (u )-\cG_{m,p}(\mathbf{u}_m ) }_{Y} + 	\norm{\cG_{m,p}(\mathbf{u}_m) -\cG_{\bN}(\mathbf{u}_m ) }_{Y} ,
	\end{eqnarray}
	where  $\cG_{m,p}(\mathbf{u}_m ) $ is defined in \eqref{eq:generic-finite-d-G-approximate} and
	\begin{eqnarray*}
		\norm{\cG (u )-\cG_{m,p}(\mathbf{u}_m ) }_{Y} 	&\leq &C h^{\alpha}+	C \omega_2(\cG (\mathcal{I}_m u ), p^{-1/d})_q,\notag \\
		\norm{\cG_{m,p}(\mathbf{u}_m) -\cG_{\bN}(\mathbf{u}_m ) }_{Y}	&\leq & C \Big(   p\sqrt{m}N_{g^{\cN}}^{-2\alpha/m} L_{g^{\cN}}^{-2\alpha/m} +   p \exp(-|\theta^{(k)}|^{\frac{1}{1+d}} )\Big),
	\end{eqnarray*}
where the constant $C$ is independent of $m$, $p$   and $N_{g^{\cN}}$ is the number of neurons in each layer of the branch network, $|\theta^{(k)}|$ is the size of the trunk network, and  $L_{g^{\cN}}$ is the numbers of layers of the branch network. 
\end{thm}

 Theorem \ref{thm:deep-o-nets-error} is obtained by applying the triangular inequality,   the H\"older continuity \eqref{eq:operator-holder-continuity}, and  Theorems \ref{thm:bochner-riesz-fourier-error} and \ref{thm:complexity-ReLU-generic} and  \ref{thm:conveg-rate-analytic-ReLU}. The proof is presented in Section \ref{chap5_proofs}.

\begin{rem}[Complexity]
	Let $\varepsilon>0$ be the error tolerance. According to Theorem \ref{thm:deep-o-nets-error}, in order to make the total error $ \norm{\cG (u )-\cG_{\bN}(\mathbf{u}_m ) }_{Y} < \varepsilon$, we need to set $m = Ch^{-d} \sim \varepsilon^{-\frac{d}{\alpha}}  $, $p \sim \varepsilon^{-\frac{d}{2}}$, $N_{g^{\cN}} \cdot L_{g^{\cN}} \sim \varepsilon^{-\frac{d}{\varepsilon}}$ and $|\theta^{(k)}| \sim \Big(\frac{d+2}{2} \ln(\frac{1}{\varepsilon}) \Big)^{d+1} $.
\end{rem}

The term  $c_k(\cG (\mathcal{I}_m u ))$ in the error estimate in Theorem \ref{thm:deep-o-nets-error} is problematic as the dimensionality of the function (of $\mathbf{u}_m$) is usually high if only H\"older continuity is assumed.  For high dimensional functions being approximated by neural networks, special architectures of neural networks are available;  see, e.g.,  
\cite{MonYang20}  and \cite{Sch20}   for H\"older continuous functions using  Kolmogorov-Arnold Superposition.  However, the convergence rate is similar.  For linear operators, we obtain  much better results.   

\subsection{Linear operators}
If $\cG$ is linear, we  can simplify the error estimates for the branch.
By the linearity of $\cG$ and the functional $c_k$,  
$c_k(\cG (\mathcal{I}_mu) )$ is  linear  in $u(x_l)$ as  
\[c_k(\cG (\mathcal{I}_m u ) )=  c_k(\cG (\sum_{l=1}^m  u(x_l) L_l(x))=\sum_{l=1}^m  u(x_l)   c_k(\cG (L_l(x)),\] where the $L_l(x)$
is the  interpolation basis, e.g., piecewise constant or piecewise polynomials.  
In this case,   DeepONets can be written as 
\begin{equation}\label{eq:deep-o-nets-linear-g}
\cG_{\bN} (\mathbf{u}_m)= \sum_{k=1}^p
\sum_{l=1}^m c_l^k u(x_l)  f^{\mathcal{N}}(  y;\theta^{(k)}) . 
\end{equation}
If  a deep network is employed to approximate the linear functions of $(u(x_1),\ldots, u(x_m))$, then we can apply  the fact that 
$z= \max{(z,0)} -\max{(-z,0)}$ for any $z\in \Real$ and obtain that 
\begin{equation}\label{eq:deep-o-nets-linear-g-ReLU}
\cG_{\bN} (\mathbf{u}_m)= \sum_{k=1}^p \Big(\text{ReLU}
\big(\sum_{l=1}^m c_l^k u(x_l) \big) -  \text{ ReLU}
\big(-\sum_{l=1}^m c_l^k u(x_l) \big)\Big)  \underbrace{f^{\mathcal{N}}(  y;\theta^{(k)})}_{trunk}. 
\end{equation}
Here $\text{ReLU} (\cdot) = \max (0, \cdot)$.

\begin{thm}
	\label{thm:error-deeponet-linear} 
	Under the same assumptions as in Theorem \ref{thm:deep-o-nets-error},  there exists a network $\cG_{\bN}(\mathbf{u}_m )$ of the form in \eqref{branchtrunkformula} for a linear operator $\cG$ such that 
	\begin{eqnarray}  
	\norm{\cG (u )-\cG_{\bN}(\mathbf{u}_m ) }_{Y} &\leq& C \Big( h  +	 \omega_2(\cG (\mathcal{I}_m u ), p^{-1/d})_q+   \exp(-\frac{1}{2}|\theta^{(k)}|^{\frac{1}{1+d}}) \Big),
	\end{eqnarray}
	where $|\theta^{(k)}| $ is  of the size $ (2\ln p)^{1+d}$ and the same notations are used as in  Theorem \ref{thm:deep-o-nets-error}. 
	\end{thm}

\begin{rem}[Complexity]
	Let $\varepsilon>0$. According to Theorem \ref{thm:error-deeponet-linear},  we need to set $m = Ch^{-d} \sim \varepsilon^{-d}  $, $p \sim \varepsilon^{-\frac{d}{2}}$, and $|\theta^{(k)}| \sim \Big(\frac{d+2}{2} \ln(\frac{1}{\varepsilon}) \Big)^{d+1} $ to have $ \norm{\cG (u )-\cG_{\bN}(\mathbf{u}_m ) }_{Y} < \varepsilon$. Furthermore, if $\mathcal{G}(\mathcal{I}_mu)$ is smooth enough such that $\sum_{k=1}^p |c_k(\mathcal{G}(\mathcal{I}_mu))| \leq C$ for some constant $C$ independent of $p$, then the error estimate for the trunk network \eqref{ineq:universal-apprx-operator} can be refined to be
	\begin{eqnarray*}
		\norm{B_R^\gamma\cG (\mathcal{I}_m u )-\widetilde{\cG}_{\bN} (\mathcal{I}_m u) }_{Y} \leq C \max_{1\leq k\leq p} \norm{e_k-  f^{\mathcal{N}}(  \cdot;\theta^{(k)})}_Y,
	\end{eqnarray*}
	which implies that $|\theta^{(k)}| $ is of  order $\Big( \ln(\frac{1}{\varepsilon}) \Big)^{d+1} $.
\end{rem}


\section{Error bounds for solution operators for Burgers and  advection-diffusion equations}\label{chap4_burgers_advection}
  In this section, we present convergence analysis 
 of DeepONets for two classes of solution operators: one is from Burgers equations, and the other is from advection-diffusion-reaction equations. In both cases, we use a semi-analytic approach to construct $\cG_{m,p}$, the finite dimensional operator approximating the operator $\cG$; see e.g., \eqref{gmpdef}.  For Burgers equations, we use the Cole-Hopf transformation (e.g., in \cite{BookAmes}) to obtain analytical solutions and then parameterize the input (initial condition).   For linear advection-diffusion equations, we use  an analytical formulation to obtain a solution in 1D before any parameterization of the input (advection coefficients). We apply finite difference schemes to obtain an approximating solution  in 2D. We show that  formulations of the resulting approximate solutions are essential to obtain realistic convergence rates of DeepONets with respect to the sizes of branch networks.

 For presentation, we only consider 
 initial condition(s) or the advection coefficients to be the input(s) of the solution operator. We also assume that the input is piecewise linear/constant function(s). For general input functions, one may obtain the error estimates immediately by combining the main results in this section, the operator's H\"older continuity, and the piecewise linear or constant interpolation error.

\subsection{1D Burgers equation with periodic boundary conditions \eqref{eqtburgers}}\label{burgers_periodic}
 Define $x^l_j = x^0_j + 2\pi l$, $j=0,1,\cdots,m$, for each $l \in \mathbb{Z}$. Then $\{ x_j^l\}_{j=0}^m$ form  a partition of $[-\pi+2\pi l, \pi+2\pi l)$. For simplicity, we denote $x_j = x^0_j$. 
To make sure that $v_0(x)$ in \eqref{burgers-ch-pbc}  is $2\pi$-periodic, we require that the initial condition has zero mean in a period
$\bar{u}_0:=\int_{-\pi}^{\pi}u_0(s)ds = 0$.
Then, by the Cole-Hopf transformation, the solution to \eqref{eqtburgers} can be written as 
\begin{eqnarray}\label{burgers-ch-pbc}
u = \frac{-2\kappa v_x}{v}, \quad \hbox{where} \quad 
\left\{\begin{array}{ll}
v_t=\kappa v_{xx},\\
v(x,0) = v_0(x) = \exp\Big(-\frac{1}{2\kappa} \int_{-\pi}^x u_0(s) ds \Big).
\end{array}\right. 
\end{eqnarray}
Since $0 < \exp(-\frac{\pi \| u_0\|_{\infty}}{\kappa}) \leq v_0(x) \leq \exp(\frac{\pi \| u_0\|_{\infty}}{\kappa})$, the solution $u$ can be written explicitly as 
\begin{eqnarray}
u(\mathbf{x}) = \mathcal{G}(u_0)(\mathbf{x})&:=&-2\kappa \frac{ \int_{\mathbb{R}} \partial_x \mathcal{K}(x,y,t) v_0(y) dy}{ \int_{\mathbb{R}} \mathcal{K}(x,y,t) v_0(y) dy},\quad \mathbf{x} = (x,t),
\label{burgersexcslt}
\end{eqnarray}
where $\mathcal{K}(x,y,t) = \frac{1}{\sqrt{4\pi\kappa t}} \exp\Big(- \frac{(x-y)^2}{4\kappa t} \Big)$ is the heat kernel. 
It can be readily checked that $u(\mathbf{x})|_{[-\pi,\pi) \times [0,\infty)}$ is the unique solution to \eqref{eqtburgers}.

We may obtain $\cG_{m,p} $ in two steps. 
The first step is to approximate $\mathcal{G}$ by $\cG_{m}$,  which is chosen to be a rational function with respect to the initial condition $v_0$. Define $\mathbf{V}_m :=  \mathbf{V}(\mathbf{u}_{0,m})=(V_0, V_1,\cdots, V_{m-1})^\top$, where $V_0=1$ and $V_j =  \exp(-\frac{u_{0,j}+u_{0,j-1}}{4\kappa}   h_j)$, $j=1,\cdots,m-1$. Define $\cG_m = \Big( \tilde{\cG}_m \circ \mathbf{V}\Big)$ as:
\begin{eqnarray}
\cG_m  (\mathbf{u}_{0,m}; \mathbf{x}) &=& \tilde{\cG}_m  (\mathbf{V}_{m}; \mathbf{x}) \nonumber\\
&=& \frac{-2\kappa \int_{\mathbb{R}} \partial_x \mathcal{K}(x,y,t) (\mathcal{I}_m^1v_0)(y) dy }{\int_{\mathbb{R}} \mathcal{K}(x,y,t)( \mathcal{I}_m^0 v_0)(y) dy}= \frac{v_0^0 c_0^1(\mathbf{x}) + v_1^0 c_1^1(\mathbf{x}) + \cdots + v_{m-1}^0 c_{m-1}^1(\mathbf{x}) }{ v_0^0 c_0^2(\mathbf{x}) + v_1^0 c_1^2(\mathbf{x}) + \cdots + v_{m-1}^0 c_{m-1}^2(\mathbf{x}) },\label{defrational}
\end{eqnarray}
where $\mathcal{I}^0_m f$ and $\mathcal{I}^1_m f$ be the piecewise constant interpolation and piecewise linear interpolation of $f$ on each sub-interval $[x_{j-1}, x_j)$, respectively, and $v_j^0 = v_0(x_j)= \prod_{i=0}^j V_i$, 
and for $j_1=1,\cdots, m-1$, $j_2=0,\cdots, m-1$, and
\begin{eqnarray*}
	c_0^1(\mathbf{x}) &=& -2\kappa \Big[ \int_{x_0}^{x_1} \Big( \sum_{l\in \mathbb{Z}} \partial_x \mathcal{K}(x,y+2\pi l, t)  \Big) \frac{x_1 - y}{x_1-x_0} dy+ \int_{x_{m-1}}^{x_m} \Big( \sum_{l\in \mathbb{Z}} \partial_x \mathcal{K}(x,y+2\pi l, t)  \Big) \frac{ y - x_{m-1}}{x_m-x_{m-1}} dy\Big], \\
	c_{j_1}^1(\mathbf{x}) &=& -2\kappa \Big[ \int_{x_{j_1-1}}^{x_{j_1}} \Big( \sum_{l\in \mathbb{Z}} \partial_x \mathcal{K}(x,y+2\pi l, t)  \Big) \frac{y - x_{j_1-1}}{x_{j_1} - x_{j_1-1}} dy+ \int_{x_{j_1}}^{x_{j_1+1}} \Big( \sum_{l\in \mathbb{Z}} \partial_x \mathcal{K}(x,y+2\pi l, t)  \Big) \frac{ x_{j_1+1}-y }{x_{j_1+1}-x_{j_1}} dy\Big];\\
	c_{j_2}^2(\mathbf{x}) &=& \int_{x_{j_2}}^{x_{j_2+1}} \Big( \sum_{l\in \mathbb{Z}} \mathcal{K}(x,y+2\pi l, t)  \Big)  dy.
\end{eqnarray*}
Hence, for any $\mathbf{x}$, $\tilde{\cG}_m$ is a rational function with respect to $\mathbf{V}_m$, where   the numerator and the denominator are both $m$-th degree $m$-variable polynomials with $m$ terms. Then we have the following error estimate.
\begin{thm}\label{thm:oprterror1d}
	Suppose that $u_0\in \mathcal{S}$ is a piecewise linear function. Let $\mathcal{G}(u_0)(\mathbf{x})$ and $\cG_m  (\mathbf{u}_{0,m}; \mathbf{x})$ be defined in \eqref{burgersexcslt} and \eqref{defrational}, respectively.  Suppose $h$ is small enough. Then there is a uniform constant $C = 2\Big(\frac{M_0^2}{\kappa } + M_1 \Big)$, such that for any $\mathbf{x}=(x,t)\in (-\pi,\pi)\times (0,+\infty)$, we   have 
	\begin{eqnarray*}
		\Big| \mathcal{G}(u_0)(\mathbf{x}) - \cG_m  (\mathbf{u}_{0,m}; \mathbf{x}) \Big| \leq Ch.
	\end{eqnarray*}
\end{thm}
In order to approximate $\cG_m$ by a neural network, it is important to realize that a rational function can be approximated by a ReLU network according to the following theorem.
\begin{thm}[\cite{telgarsky2017neural}]\label{thm:rational-apprx-ReLU}
	Let $\varepsilon \in (0,1]$ and nonnegative integer $k$ be given. Let $p:[0,1]^d\rightarrow [-1,1]$ and $q:[0,1]^d\rightarrow [2^{-k},1]$ be polynomials of degree $\leq r$, each with $\leq s$ monomials. Then there exists
	a ReLU network $f$ of size (number of total neurons)
	\begin{eqnarray}
	\mathcal{O}\Big(k^7 \ln(\frac{1}{\varepsilon})^3 + \min\{ srk\ln(sr/\varepsilon), \ sdk^2\ln(dsr/\varepsilon)^2 \} \Big)
	\label{lm11}
	\end{eqnarray}
	such that 
	\begin{eqnarray}
	\sup_{x\in[0,1]^d} \Big| f(x) - \frac{p(x)}{q(x)}\Big| \leq \varepsilon.
	\label{lm12}
	\end{eqnarray} 
\end{thm}
Combining Theorem \ref{thm:oprterror1d} and Theorem \ref{thm:rational-apprx-ReLU} readily leads  to  Theorem \ref{thm:burgersoprtest1}, which is used to show the error estimate for branch networks.

\begin{thm}\label{thm:burgersoprtest1}
Let $u_0\in \mathcal{S}$ be a piecewise linear function. Let $\mathcal{G}(u_0)$ be the solution operator of the Burgers equation  \eqref{eqtburgers}. 
	Then there exist a ReLU network  $g^{\mathcal{N}}(\mathbf{u}_{0,m}; \Theta)$ of size $\abs{\Theta}=\mathcal{O}(m^2\ln(m) )$ and a uniform constant $C=C(\kappa, M_0, M_1)$, such that for any $\mathbf{x}\in [-\pi,\pi)\times(0,\infty)$, there exists a set of parameters $\Theta_{\mathbf{x}}$, s.t.
	\begin{eqnarray*}
		\Big| \mathcal{G}(u_0)(\mathbf{x}) - g^{\mathcal{N}}(\mathbf{u}_{0,m}; \Theta_{\mathbf{x}})\Big| 
	\leq \Big| \mathcal{G}(u_0)(\mathbf{x}) - \cG_m  (\mathbf{u}_{0,m}; \mathbf{x}) \Big|  + 
	\Big|  \cG_m  (\mathbf{u}_{0,m}; \mathbf{x})- g^{\mathcal{N}}(\mathbf{u}_{0,m}; \Theta_{\mathbf{x}}) \Big|	 \leq Ch.
	\end{eqnarray*}
\end{thm}

\begin{rem}
	Let  $u_0(x) = \sum_{i=0}^{\mathsf{m}-1} u_i L_i(x)$ be a  piecewise   function, where $\mathsf{m} $ is a fixed number and $\mathsf{m} \neq m$ generally. According to the asymptotic result in Theorem \ref{thm:burgersoprtest1}, we still need  $m\to \infty$ to guarantee  convergence. However, since $\mathsf{m} \ll m$ when $m\rightarrow \infty$, one can combine like terms (in terms of $u_i$) in \eqref{defrational}, so that there are only $\mathsf{m}$ variables as the inputs of the rational function $\cG_m$. Then by \eqref{lm11} again, the size of the network can be reduced to $\mathsf{m} \cdot \mathcal{O}(m \ln(m)^2)$, where $d=\mathsf{m}$, $s=m$ and $r=m$. The same remark can be applied to all of the examples in this section. 
\end{rem}

Theorem \ref{thm:main_burgers} follows from the above theorems. 
 The detailed proofs of Theorem \ref{thm:oprterror1d}, Theorem \ref{thm:burgersoprtest1} and Theorem \ref{thm:main_burgers} are presented in  Section \ref{chap5_proofs}.

\begin{rem}
	Solution operators from a large class of such PDEs can be considered.	Many nonlinear PDEs  can be converted to linear problems by some transformations, such as 
	the Cole-Hopf transformation, the hyperbolic ansatz, and the improved tanh-coth method, see, e.g.,  \cite{BookAmes, polyanin2011handbook}.

\end{rem} 

\subsection{1D Burgers equation with Dirichlet boundary condition}
Consider the Burgers equation with Dirichlet boundary condition $u(a,t)=F_0(t)$ and 
$u(b,t)=F_1(t)$ where $-\infty<a  < b<\infty$. As in \cite{RanChang91}, the Burgers equation can be transformed into the heat equation with the Neumann/Robin boundary condition by the Cole-Hopf transformation, i.e.,
\begin{eqnarray}\label{burger-ch-dbc}
u = \frac{-2\kappa v_x}{v}, \quad \hbox{where} \quad 
\left\{\begin{array}{ll}
v_t=\kappa v_{xx}, \quad x\in (a,b),\\
v(x,0) = v_0(x) = \exp\Big(-\frac{1}{2\kappa} \int_{a}^x u_0(s) ds \Big),\\
2\kappa v_x(a,t) + F_0(t)v(a,t)  =  2\kappa v_x(b,t) + F_1(t)v(b,t)=0,
\end{array}\right. 
\end{eqnarray}
which admits a solution of the following form
\begin{equation*}
v(x,t) = \int_{\Real}\mathcal{K}(x-y,t) \phi_0(y)\,dy - \int_0^t \mathcal{K}(x-a,t-s)\Psi_1(s)\,ds +\int_0^t \mathcal{K}(x-b,t-s)\Psi_2(s)\,ds,
\end{equation*}
where $\Psi_1(t)$ and $\Psi_2(t)$ depend linearly on the initial condition $v_0$ (or $\{ v_j^0\}_{j=0}^{m-1}$), see  \cite{RanChang91} for details. 
The construction of $\cG_m$ in this case  is similar to   \eqref{defrational}. and hence, similar results to those in Section \ref{burgers_periodic} can be obtained.

\subsection{1D Burgers equations with forcing}
Consider    the  following Burgers equation  with a forcing term
\begin{eqnarray} 
u_t + u  u_x = \kappa u_{xx}+f(x,t), \ \ x \in \mathbb{R}, 
\end{eqnarray}
with initial condition $u_0(x)$. By the Cole-Hopf transformation, we obtain a linear heat equation: 
\begin{eqnarray}\label{burgers-ch-forcing}
\quad u = \frac{-2\kappa v_x}{v}, \quad \hbox{where} \quad 
\left\{\begin{array}{ll}
v_t= \kappa v_{xx}-F(x,t)v, \  F(x,t) =\frac{1}{2\kappa} \int_{-\infty}^x f(y,t)dy,\ x\in \mathbb{R},\\
v(x,0) = v_0(x) = \exp\Big(-\frac{1}{2\kappa} \int_{-\infty}^x u_0(s) ds \Big).
\end{array}\right. 
\end{eqnarray}
Assume that  $F(x,t)$ has a lower bound and 
$
\int_{0}^t \sup_{x} \abs{F(x,t)}\,ds<\infty$.  
By the Feynman-Kac formula,  
\begin{eqnarray} 
v(x,t) =\mean{\exp\big(- \int_{0}^t F(x+ \sqrt{2\kappa}B_s,s)\,ds\big) v_0(x+\sqrt{2\kappa}B_t)},
\end{eqnarray}
where $B_t$ is a standard Brownian motion with $B_0=0$. And by a direct calculation, we have 
\begin{equation*}
v_x = \mean{\frac{-1}{2\kappa}\exp\big(- \int_{0}^t F(x+ \sqrt{2\kappa}B_s,s)\,ds\big) [ u_0(x+\sqrt{2\kappa}B_t)+  \int_{0}^t f(x+ \sqrt{2\kappa}B_s,s)\,ds]v_0(x+\sqrt{2\kappa}B_t)}.
\end{equation*}
Suppose that $u_0(x)$ is   a piecewise linear polynomial and $f(x,t)$ can be well approximated by piecewise linear or piecewise constant polynomials, i.e., 
$u_0(x) =\sum_{i=1}^I u_{0,i}L_i(x)$, $f(x,t) = \sum_{j=1}^J f(x_j,t_j) e_{j}(x,t)$. 
Denote 
$ \tilde{e}_j^* = \int_{0}^{t^*} \int_{-\infty}^{x^*} e_j(y+ \sqrt{2\kappa}B_s,s)\,dy\,ds $ and 
$\tilde{L}_i^*=\int_{-\infty}^{x^*+\sqrt{2\kappa}B_{t^*}} L_i(y)\,dy$. Then 
\begin{eqnarray*} 
	v(x^*,t^*) 
	&=&  \mean{\exp\big(- \sum_{j=1}^J \frac{f(x_j,t_j)}{2\kappa } \tilde{e}_j^* -\sum_{i=1}^I\frac{u_{0,i}}{2\kappa} \tilde{L}_i^*\big)}  + \mathcal{O}(h_x+h_t) \\
	&=&  \frac{1}{N}\sum_{l=1}^N \Big( \exp\big(- \sum_{j=1}^J \frac{f(x_j,t_j)}{2\kappa} \tilde{e}_j^*   (\omega_l) -\sum_{i=1}^I\frac{u_{0,i}}{2\kappa} \tilde{L}_i^*(\omega_l)\Big)  + \mathcal{O}(N^{-1/2})+ \mathcal{O}(h_x+h_t) .  
\end{eqnarray*}
Here $\tilde{e}_j^*   (\omega_l)$ ($\tilde{L}_i^*(\omega_l)$) is a  realization of $\tilde{e}_j^*  $ ($\tilde{L}_i^*$) at one  trajectory. 
Similarly, we have 
\begin{eqnarray*}
	v_x (x^*,t^*)&= & -\frac{1}{2\kappa}\mean{\exp\big(- \sum_{j=1}^J \frac{f(x_j,t_j)}{2\kappa } \tilde{e}_j^* -\sum_{i=1}^I\frac{u_{0,i}}{2\kappa} \tilde{L}_i^*\big) (  u_0(x^*+\sqrt{2\kappa}B_{t^*}) +\int_0^{t^*}f(x^*+\sqrt{2\kappa}B_s),s)\,ds} \\
	&& + \mathcal{O}(h_x + h_t )\\
	&=& -\frac{1}{2\kappa} \mean{\exp\big(- \sum_{j=1}^J \frac{f(x_j,t_j)}{2\kappa } \tilde{e}_j^* -\sum_{i=1}^I\frac{u_{0,i}}{2\kappa} \tilde{L}_i^*\big) (  \sum_{i=1}^I u_{0,i}L_i^* +\sum_{j=1}^J f(x_j,t_j) \bar{e}_j^* } \\
	&& +\mathcal{O}(h_x   +  h_t ) \\
	&=& \mean{\exp\big(- \sum_{j=1}^J \frac{f(x_j,t_j)}{2\kappa } \tilde{e}_j^* -\sum_{i=1}^I\frac{u_{0,i}}{2\kappa} \tilde{L}_i^*\big) \times \\
		&& \big(  \frac{ \exp(\sum_{i=1}^I \frac{u_{0,i}L_i^*}{2\kappa} h_x)-1)}{h_x} +
		\frac{\exp(\sum_{j=1}^J f(x_j,t_j) \bar{e}_j^*h_t)-1 }{h_t}\big)}  
	  +\mathcal{O}(h_x+ h_t )\\
	&=&\frac{1}{N} \sum_{l=1}^N\exp\big(- \sum_{j=1}^J \frac{f(x_j,t_j)}{2\kappa } \tilde{e}_j^* -\sum_{i=1}^I\frac{u_{0,i}}{2\kappa} \tilde{L}_i^*\big) (\omega_l)\times \\
	&& \big(  \frac{ \exp(\sum_{i=1}^I \frac{u_{0,i}L_i^*}{2\kappa} h_x)-1)}{h_x} +
	\frac{\exp(\sum_{j=1}^J f(x_j,t_j) \bar{e}_j^*h_t)-1 }{h_t}\big) (\omega_l)\\
	&& +\mathcal{O}(h_x )   + \mathcal{O}( h_t)   + \mathcal{O}( N^{-1/2}) 
\end{eqnarray*}
Here $L_i^*=L_i(x^*+\sqrt{2\kappa}B_{t^*})$, $\bar{e}_j^*= \int_0^{t^*}f(x^*+\sqrt{2\kappa}B_s),s)\,ds$. 

Combining the above estimates and  
by Lemma \ref{lem:rtnfuncerr}, we have 
\begin{eqnarray*} 
	&& u (x^*,t^*) 
	= -\frac{1}{2\kappa} \frac{v_x (x^*,t^*)}{	v (x^*,t^*)}\\
	&=&  -\frac{1}{2\kappa} \sum_{l=1}^N  
	\frac{\exp\big(- \sum_{j=1}^J \frac{f(x_j,t_j)}{2\kappa } \tilde{e}_j^* -\sum_{i=1}^I\frac{u_{0,i}}{2\kappa} \tilde{L}_i^*\big) (\omega_l) \big(  \frac{ \exp(\sum_{i=1}^I \frac{u_{0,i}L_i^*}{2\kappa} h_x)-1)}{h_x} +
		\frac{\exp(\sum_{j=1}^J f(x_j,t_j) \bar{e}_j^*h_t)-1 }{h_t}\big) (\omega_l) }
	{ \sum_{l=1}^N \Big( \exp\big(- \sum_{j=1}^J \frac{f(x_j,t_j)}{2\kappa} \tilde{e}_j^*   (\omega_l) -\sum_{i=1}^I\frac{u_{0,i}}{2\kappa} \tilde{L}_i^*(\omega_l)\Big) } \\
	&&+ \mathcal{O}(N^{-1/2})+ \mathcal{O}(h_x + h_t ).
\end{eqnarray*}	

Similar to the treatments of Burgers equation without forcing term, we replace 
$\exp\big(- \sum_{j=1}^J \frac{f(x_j,t_j)}{2\kappa} \tilde{e}_j^*   (\omega_l)$ with 
$\prod_{j=1}^J  (1-    \sum_{j=1}^J \frac{f(x_j,t_j)}{2\kappa} \tilde{e}_j^*   (\omega_l) )$, which leads to an error of $\mathcal{O}(h_x+h_t)$. 
The solution becomes a summation of $N$ rational polynomials, each of which 
has $N$ terms in the denominator with  all the term of order $I+J$, $J= I_x\times I_t$ and $I=I_x$. Then by Theorem \ref{thm:rational-apprx-ReLU} and similar to the proof of Theorem \ref{thm:burgersoprtest1}, we have 
a ReLU network of size 
$\mathcal{O}(N(I+J)^2\ln (I+J)/\epsilon)$ to approximate $u$.

Taking $\epsilon\simeq I_x^{-1} \simeq  I_t^{-1} \approx N^{-1/2}$. Then   the network size is 
$\mathcal{O}(\epsilon^{-6}\ln(\epsilon^{-3}))$.
If we do not introduce the discretization in $t$, then $J=I=I_x$ and 
the network size can be reduced to  
$\mathcal{O}(\epsilon^{-4}\ln(\epsilon^{-2}))$. 
If we  use alternative numerical integration methods, e.g.,  quasi-Monte Carlo methods, then  
$N=\epsilon^{-1}$ and   the network size can be reduced to  
$\mathcal{O}(\epsilon^{-3}\ln(\epsilon^{-2}))$. 

\begin{rem}
	The condition  we need in the Cole-Hopf transformation is 
	$\int_{\Real} \exp(-\epsilon x^2 ) \abs{v_0(x)}\,dx <\infty$, which can be satisfied when $u_0(x)$ is a piecewise linear polynomial as assumed. Here 
	$\epsilon>0$.
\end{rem}

\subsection{2D Burgers equation}
Consider the 2D Burgers equation with periodic boundary condition:
\begin{eqnarray}\label{eqtburgers2d}
\left\{\begin{array}{l}
u_{t}+u u_{x}+v u_{y}-\kappa \left(u_{x x}+u_{y y}\right)=0, \quad (x,y)\in (-\pi,\pi)\times(-\pi,\pi),  \\
v_{t}+u v_{x}+v v_{y}-\kappa \left(v_{x x}+v_{y y}\right)=0, \quad (x,y)\in (-\pi,\pi)\times(-\pi,\pi), \\
u(x-\pi,y) = u(x+\pi,y), \ u(x,y-\pi) = u(x,y+\pi),\\
v(x-\pi,y) = v(x+\pi,y), \ v(x,y-\pi) = v(x,y+\pi),\\
u(x,y,0) = u_0(x,y),\\ 
v(x,y,0) = v_0(x,y).\\ 
\end{array}\right.
\end{eqnarray}
where $u_0$ and $v_0$ satisfy the consistent condition $\partial_y u_0 = \partial_x v_0=w_0$, i.e., they are of the forms 
\begin{eqnarray*}
u_0(x,y) = \int_{-\pi}^y w_0(x,s)ds + \tilde{u}_0(x) \quad \hbox{and} \quad v_0(x,y) =  \int_{-\pi}^x w_0(r,y)dr + \tilde{v}_0(y).
\end{eqnarray*}
 Then we may apply the Cole-Hopf transformation, see e.g., \cite{Fle83},
\begin{eqnarray}\label{burgersexcslt2d}
u=\mathcal{G}^u(u_0,v_0):= -2\kappa \frac{\phi_x}{\phi}, \quad v=\mathcal{G}^v(u_0,v_0):= -2\kappa \frac{\phi_y}{\phi},  
\end{eqnarray}
where
\begin{equation*}
\left\{\begin{array}{l}
\phi_t - \kappa \Delta \phi =0,\\
\phi_0(x,y) = \exp\Big( -\frac{1}{2\kappa} \Big[\int_{-\pi}^x \int_{-\pi}^y w_0(s,r)drds + \int_{-\pi}^x \tilde{u}_0(s)ds + \int_{-\pi}^y \tilde{v}_0(r) dr \Big]\Big).
\end{array}\right.
\end{equation*}
Define
\begin{eqnarray*}
\mathcal{S} = \mathcal{S}(M_0,M_1):=\{ (u_0,v_0) \in W^{1,\infty}\times W^{1,\infty}: \| \phi_0\|_{L^{\infty}} \leq M_0, \| \nabla \phi_0\|_{L^{\infty}} \leq M_1 \}.
\end{eqnarray*}
Let $\mathcal{I}^0_m$ and $\mathcal{I}^1_m$ be the piecewise constant interpolation at the left bottom corner and bilinear interpolation on each rectangular element $[x_{j_1-1}, x_{j_1})\times [x_{j_2-1}, x_{j_2})$, $j_1, j_2 = 1,2,\cdots, \sqrt{m}$, respectively. Denote $\mathbf{x} = (x,y)^\top$ and $\mathbf{s}=(s,r)^\top$. Then we define
\begin{eqnarray}\label{rationaloprt2d}
\cG_m^u(\mathbf{u}_{m,0},\mathbf{v}_{m,0};\mathbf{x},t)&:=&-\frac{2\kappa \int_{\mathbb{R}} \int_{\mathbb{R}} \partial_x \mathcal{K}(\mathbf{x},\mathbf{s},t) (\mathcal{I}^1_m\phi_0)(\mathbf{s}) dsdr }{\int_{\mathbb{R}}\int_{\mathbb{R}} \mathcal{K}(\mathbf{x},\mathbf{s},t)( \mathcal{I}^0_m \phi_0)(\mathbf{s}) dsdr}, \nonumber \\
\cG_m^v(\mathbf{u}_{m,0},\mathbf{v}_{m,0};\mathbf{x},t)&:=&-\frac{2\kappa \int_{\mathbb{R}} \int_{\mathbb{R}} \partial_y \mathcal{K}(\mathbf{x},\mathbf{s},t) (\mathcal{I}^1_m\phi_0)(\mathbf{s}) dsdr }{\int_{\mathbb{R}}\int_{\mathbb{R}} \mathcal{K}(\mathbf{x},\mathbf{s},t)( \mathcal{I}^0_m \phi_0)(\mathbf{s}) dsdr}.
\end{eqnarray}
where $ \mathcal{K}(\mathbf{x},\mathbf{s},t) = \frac{1}{4\pi \kappa t} \exp\Big(-\frac{|\mathbf{x}-\mathbf{s}|^2}{4\kappa t} \Big)$ is the 2-D heat kernel. Then we may derive the following results similar to the 1-D case \eqref{eqtburgers}.
\begin{lem}\label{cor:2dburgersbrancherro-apprx}
	Let $(u_0,v_0)\in \mathcal{S}$. Let $\mathcal{G}^u$, $\mathcal{G}^v$ be the solution operators for \eqref{eqtburgers2d} defined in \eqref{burgersexcslt2d}, and $\cG_m^u$, $\cG_m^v$ be defined in \eqref{rationaloprt2d}. Suppose $h = \max_j |x_j - x_{j-1}|$ is small enough. Then there is a uniform constant $C = C(\kappa, M_0, M_1)$, s.t. for any $(\mathbf{x},t)\in (-\pi,\pi)^2\times (0,\infty)$, we  have 
	\begin{eqnarray*}
		\Big| \mathcal{G}^u(u_0,v_0)(\mathbf{x},t) - \cG_m^u  (\mathbf{u}_{0,m},\mathbf{v}_{0,m}; \mathbf{x},t) \Big|+\Big| \mathcal{G}^v(u_0,v_0)(\mathbf{x},t) - \cG_m^v  (\mathbf{u}_{0,m},\mathbf{v}_{0,m}; \mathbf{x},t) \Big| \leq Ch.
	\end{eqnarray*}
\end{lem}
 
\begin{lem}\label{cor:2dburgersbrancherror}
	Let $(u_0,v_0)\in \mathcal{S}$. Let $\mathcal{G}^u$, $\mathcal{G}^v$ be the solution operators for \eqref{eqtburgers2d} defined in \eqref{burgersexcslt2d}. Then exist ReLU networks 
	$g^{\mathcal{N}}_u(\mathbf{u}_{0,m}, \mathbf{v}_{0,m}; \Theta_u)$
	 of size $\abs{\Theta_u}=\mathcal{O}(m^2 \ln(m) )$ and
	 $g^{\mathcal{N}}_v(\mathbf{u}_{0,m}, \mathbf{v}_{0,m}; \Theta_v)$
	 of size $\abs{\Theta_v}=\mathcal{O}(m^2 \ln(m) )$ and a uniform constant $C=C(\kappa, M_0, M_1)$, such that for any $(\mathbf{x},t)\in (-\pi,\pi)^2\times(0,\infty)$, there exist parameters $\Theta_u^{(\mathbf{x},t)}$ and $\Theta_v^{(\mathbf{x},t)}$, s.t.
	\begin{eqnarray*}
		\Big| \mathcal{G}^u(u_0,v_0)(\mathbf{x},t) - g^{\mathcal{N}}_u(\mathbf{u}_{0,m}, \mathbf{v}_{0,m}; \Theta_u^{(\mathbf{x},t)})\Big| + \Big| \mathcal{G}^v(u_0,v_0)(\mathbf{x},t) - g^{\mathcal{N}}_v(\mathbf{u}_{0,m}, \mathbf{v}_{0,m}; \Theta_v^{(\mathbf{x},t)})\Big| \leq Ch.
	\end{eqnarray*}
\end{lem}
\begin{proof}
 Observe that $\cG_m^u$ and $\cG_m^v$ are rational functions with respect to $\phi_0(x_{j_1}, x_{j_2})$ with $r=m$ and $s = m$ and $k$ is a constant (using the notation in Theorem \ref{thm:rational-apprx-ReLU}). The conclusion then follows from Theorem 
 \ref{thm:rational-apprx-ReLU}. 
\end{proof}

 \begin{thm}
 	Let $(u_0,v_0)\in \mathcal{S}$. Let $\mathcal{G}^u$, $\mathcal{G}^v$ be the solution operators for \eqref{eqtburgers2d}. Then exist ReLU branch networks 
 	$g^{\mathcal{N}}_u(\mathbf{u}_{0,m}, \mathbf{v}_{0,m}; \Theta_u^{(i)})$ 
 	and
 	$g^{\mathcal{N}}_v(\mathbf{u}_{0,m}, \mathbf{v}_{0,m}; \Theta_v^{(i)})$ of sizes $|\Theta_u^{(i)}|$, $|\Theta_v^{(i)}|=\mathcal{O}(m^2\ln(m))$, $i=1,\cdots, p$, respectively, 
 	and  ReLU  trunk networks $f ^{\mathcal{N}}( x ;\theta^{(k)})$ of size $ \mathcal{O}(1)$,   $k=1,\cdots, p$, such that	
 	\begin{eqnarray*}
 		&&	\norm{\cG^u (u_0,v_0)-\cG^u_{\bN}(\mathbf{u}_{0,m},\mathbf{v}_{0,m} ) }_{L^{\infty}} + \norm{\cG^v (u_0,v_0 )-\cG^v_{\bN}(\mathbf{u}_{0,m}, \mathbf{v}_{0,m}) }_{L^{\infty}}\\
 		&& \leq C \Big( p^{-\frac{1}{2}} + m^{-\frac{1}{2}} + \abs{\Theta_{u}^{(i)}}^{-\frac{1}{2}+\epsilon}  + \abs{\Theta_{v}^{(i)}}^{-\frac{1}{2}+\epsilon}\Big),
 	\end{eqnarray*}
 	where  $\cG^v_{\bN}(\mathbf{u}_{0,m},\mathbf{v}_{0,m} )$  and 
 	$\cG^u_{\bN}(\mathbf{u}_{0,m},\mathbf{v}_{0,m} )$ are  of the form in \eqref{branchtrunkformula}, $\epsilon>0$ is arbitrarily small and $C >0$ is independent of $m$, $p$, $|\Theta_u^{(i)}|$, $|\Theta_v^{(i)}|$, $u_0$ and $v_0$.
 \end{thm}

\subsection{1D steady advection-diffusion equation with Dirichlet boundary condition}
\label{ssec:1dadv}
 First, we transform \eqref{1dadvection} into the divergence form:
\begin{eqnarray}\label{1dadvectiondivform}
-\partial_x (A(x) \partial_x u)=A(x)f(x), \quad \hbox{where} \quad A(x):=\exp\Big(-\int_0^x a(s)ds \Big),
\end{eqnarray}
and $0< \exp(-L\| a\|_{\infty}) \leq A(x), A^{-1}(x) \leq  \exp(L \| a\|_{\infty})$. Notice that the analytic solution to \eqref{1dadvectiondivform} can be found by the following steps: a) integrate  both sides, which generates a constant $C$ to be determined, b) divide both sides by $-A(x)$ and integrate again, c) determine $C$ by using the boundary condition. If we define the two operators:
\begin{eqnarray*}
	\mathcal{A}_+ (g)(x):= \int_0^x A(y)g(y)dy, \quad \mathcal{A}_- (g)(x):=\int_0^x A^{-1}(y)g(y)dy,
\end{eqnarray*}
then for any given $f(x)$, the analytic solution $u(x)$ can be given by the following expression symbolically:
\begin{eqnarray}\label{1dellipticexcslt}
u(x)= \mathcal{G}^f(a)(x):=-  \Big( \mathcal{A}_- \circ  \mathcal{A}_+ \Big)(f)(x) + \frac{\mathcal{A}_- (\mathbbm{1})(x)}{\mathcal{A}_- (\mathbbm{1})(L)} \Big( \mathcal{A}_- \circ  \mathcal{A}_+ \Big)(f)(L),
\end{eqnarray}
where $\mathbbm{1}(x)=1$ is the constant function. Clearly, $\mathcal{A}_{\pm} (\mathbbm{1})$ is linear w.r.t. $A^{\pm 1}(x)$ and $0\leq \frac{\mathcal{A}_-^N (\mathbbm{1})(x)}{\mathcal{A}_-^N (\mathbbm{1})(L)} \leq 1$ is increasing with respect to $x$. Define
\begin{eqnarray}\label{ellipticrationaldef}
\cG_m^f(\mathbf{a}_m;x):=-\Big( \mathcal{A}_-^N \circ  \mathcal{A}_+^N \Big)(f)(x) + \frac{\mathcal{A}_-^N (\mathbbm{1})(x)}{\mathcal{A}_-^N (\mathbbm{1})(L)} \Big( \mathcal{A}_-^N \circ  \mathcal{A}_+^N \Big)(f)(L),
\end{eqnarray}
where $\mathbf{a}_m=(a_1,\cdots, a_m)^\top$ and 
\begin{eqnarray*}
	\mathcal{A}_+^N (g)(x):= \int_0^x \mathcal{I}_m^0 (Ag) (y)dy, \quad
	\mathcal{A}_-^N (g)(x):=\int_0^x  \mathcal{I}_m^0 (A^{-1}g)(y)dy,
\end{eqnarray*}
and $\mathcal{I}_m^0 f$ represents the piecewise constant interpolation of $f$ on each subinterval $[x_{i-1}, x_i)$. Notice that $\Big( \mathcal{A}_-^N \circ  \mathcal{A}_+^N \Big)(f)$ is bilinear w.r.t. $\frac{A(x)}{A(y)}$ and $f(x)$. Then we can prove the following Lemmas.
\begin{lem}\label{prop:1dadvecrationalerr}
	Let $h>0$ be small enough. There is a constant $C = C(M_0)>0$  such that 
	\begin{eqnarray*}
		\|\mathcal{G}^f(a)(\cdot) - \cG_m^f(\mathbf{a}_m;\cdot) \|_{L^{\infty}(0,L)} \leq Ch.
	\end{eqnarray*}
\end{lem}
\begin{thm}\label{thm:1dadvecnetworkerr}
Let $a\in \Sigma$ be piecewise constant. For any given $f$, let $\cG^f(a)$ be the solution operator to the advection-diffusion equation \eqref{eqtadv}. Then there exist  ReLU networks $g^{\cN}_f(\mathbf{a}_m;\Theta)$ of size $|\Theta| = m^4 \ln(m)$ and a uniform constant $C=C(M_0)$, such that for any $x \in [0,L]$, there exists a set of parameters $\Theta_{x}$, s.t.
\begin{eqnarray*}
\abs{\cG^f(a)(x) - g^{\cN}_f(\mathbf{a}_m;\Theta_x)} \leq   \abs{ \mathcal{G}^f(a)(\cdot) - \cG_m^f(\mathbf{a}_m;x)} + \abs{ \cG_m^f(\mathbf{a}_m;x) -   g^{\cN}_f(\mathbf{a}_m;\Theta_x)}\leq Ch.
\end{eqnarray*}
\end{thm}

Finally, Theorem \ref{thm:main_1dadvec} follows from Theorem \ref{thm:1dadvecnetworkerr} and similar arguments in the proof of Theorem \ref{thm:main_burgers}.

\begin{rem}
The analysis in this subsection also works for the variable coefficient second order differential equation $-\partial_x (a(x) \partial_x u)= f(x)$, where $f(x) \in L^{\infty}([0,L])$ and $0<m \leq a(x) \leq M$. 
\end{rem}

 
\subsection{2D steady reaction-diffusion equations}\label{sec:2d-reaction-diffusion}
Consider the 2D steady reaction-diffusion equation with a classical boundary condition:
\begin{eqnarray}\label{2dreactionpde}
\left\{\begin{array}{ll}
-\Delta u(x,y) + a(x,y) u(x,y)= f(x,y), \quad \hbox{in} \ \Omega,\\
\mathcal{B}u(x,y) = 0, \quad \hbox{on} \ \partial \Omega,
\end{array}\right.
\end{eqnarray}
where $\Omega$ is a rectangular domain and $\mathcal{B}$ can be the Dirichlet, Neumann or Robin boundary operator. Define
\begin{eqnarray*}  
\mathcal{S} = \mathcal{S}(C_0) = \{a(x,y) \in L^{\infty}(\Omega): 0 \leq a(x,y) \leq C_0\}, \quad C_0>0.
\end{eqnarray*}
Suppose $f(x,y) \in L^{\infty}$ and $a(x,y)\in \mathcal{S}$. Similar to the case of Equation \eqref{1dadvection}, we only consider $a(x,y)$ as the input to the solution operator and  denote $u(x,y) = \cG^f(a(x,y))$.
Instead of using analytical solutions, we use  numerical solutions such as those from central finite  difference schemes for \eqref{2dreactionpde}. Assume that the linear system associated with the central finite difference scheme  is of the form 
\begin{eqnarray}\label{ls2dreactionpde0}
\Big( S+\Lambda \Big) U_N = F,
\end{eqnarray}
where $S$ is the stiffness matrix, $\Lambda$ is the diagonal matrix depending on the variable coefficient $a(x,y)$, $F$ is the discretized right hand side, and $U_N$ is the vector of numerical solution. 
We may  write 
$\Lambda = \sum_{i=1}^m \Lambda_i$, where $\Lambda_i=h^2 a_i I_i$, $I_i = e_i e_i^\top$, and $e_i \in \mathbb{R}^m$ is the $i$-th unit vector, 
$a_i$ is the $i$-th diagonal entry of $\Lambda$,
$h \sim m^{-\frac{1}{2}}$ is the mesh size of the numerical scheme. 
Then \eqref{ls2dreactionpde0} can be written as 
\begin{eqnarray}\label{ls2dreactionpde1}
\Big( S+ \sum_{i=1}^m h^2 a_i I_i  \Big) U_N = F.
\end{eqnarray}
Let $\mathbf{a}_m = [a_1,\cdots,a_m]^\top$. Then the numerical solution operator can be defined as
\begin{eqnarray*}
U_N = \cG_N^F(\mathbf{a}_m):= \Big( S+ \sum_{i=1}^m h^2 a_i I_i  \Big)^{-1} F.
\end{eqnarray*}
Define the exact solution of $u$ at grid points associated with the finite difference schemes by 
\begin{eqnarray}\label{defU}
U=[u(x_{1}, y_{1}), \cdots, u(x_{1}, y_{\sqrt{m}}), u(x_{2}, y_{1}),\cdots, u(x_{\sqrt{m}}, y_{\sqrt{m}})]^\top \in \mathbb{R}^m.
\end{eqnarray}
We assume that for the central difference scheme used above for Equation \eqref{2dreactionpde},  there exist constants $C>0$ and $0<r \leq 2$, such that
\begin{eqnarray}
\| U - U_N \|_{l^{\infty}} \leq Ch^r \sim m^{-\frac{r}{2}}.
\end{eqnarray} 
See e.g., \cite{JovSuli-B14}  for more details of the specific value of $r$.

In the following text, we \textit{apply  Sherman-Morrison's formula to represent $U_N$ and then construct a blessed representation \cite{mhaskar2016deep} for  branch networks}.  Recall the Sherman-Morrison formula 
\begin{eqnarray}\label{shermanmorrison}
(A+uv^\top)^{-1} = A^{-1} - \frac{A^{-1} uv^\top A^{-1}}{1+v^\top A^{-1} u},
\end{eqnarray}
where $A\in \mathbb{R}^{m\times m}$ is invertible, $u,v \in \mathbb{R}^m$ such that $1+v^\top A^{-1} u\neq 0$. Since each matrix $h^2 a_i I_i$ has rank one, \eqref{shermanmorrison} can be applied to find the recurrence formula of $(S+\Lambda)^{-1}$ explicitly. Define $T_0 := S^{-1}$, and $T_k := \Big( S+\sum_{i=1}^k h^2 a_i I_i\Big)^{-1}$ for $k=1,\cdots, m$. Thus, $T_m = (S+\Lambda)^{-1}$. Define the map  $G_k: \mathbb{R}\times \mathbb{R}^{m\times m} \rightarrow \mathbb{R}^{m\times m}$ by 
\begin{eqnarray*}
G_k(\alpha, M) := M - \frac{\alpha}{1+\alpha (e_k^\top M e_k)} M e_k e_k^\top M, \quad k=1,2,\cdots,m,\quad 1+\alpha (e_k^\top M e_k)\neq 0.
\end{eqnarray*}
By  the Sherman-Morrison formula, we have 
\begin{eqnarray}\label{SMdiscretize}
\quad \quad T_k =  \Big( T_{k-1}^{-1} + h^2 a_k I_k \Big)^{-1} = T_{k-1} - \frac{h^2 a_k}{1+h^2 a_k (e_k^\top T_{k-1} e_k)} T_{k-1} e_k e_k^\top T_{k-1} = G_k(h^2 a_k, T_{k-1}).
\end{eqnarray}
Then $(S+\Lambda)^{-1}=T_m $ can be written as  
\begin{eqnarray}\label{operatorstructure}
T_m = G_m(h^2 a_m, G_{m-1}(h^2 a_{m-1}, G_{m-1}(\cdots G_1(h^2 a_1, T_0)\cdots))).
\end{eqnarray}

According  to \cite{mhaskar2016deep}, we may utilize the compositional structure of this discrete  operator to overcome the curse of dimensionality (in $m$). 
The  neural network can be written as  $\cN_G:=\cN_{G_m} \circ \cN_{G_{m-1}} \circ \cdots \circ \cN_{G_1}$, where $\cN_{G_k}$ is a neural network approximating $G_k$, $k=1,2,\cdots,m$.  The structure of $\cN_G$ is presented in Figure \ref{ngstructure}.
  \begin{figure}[htbp] 
    \centering
    \includegraphics[width=4in]{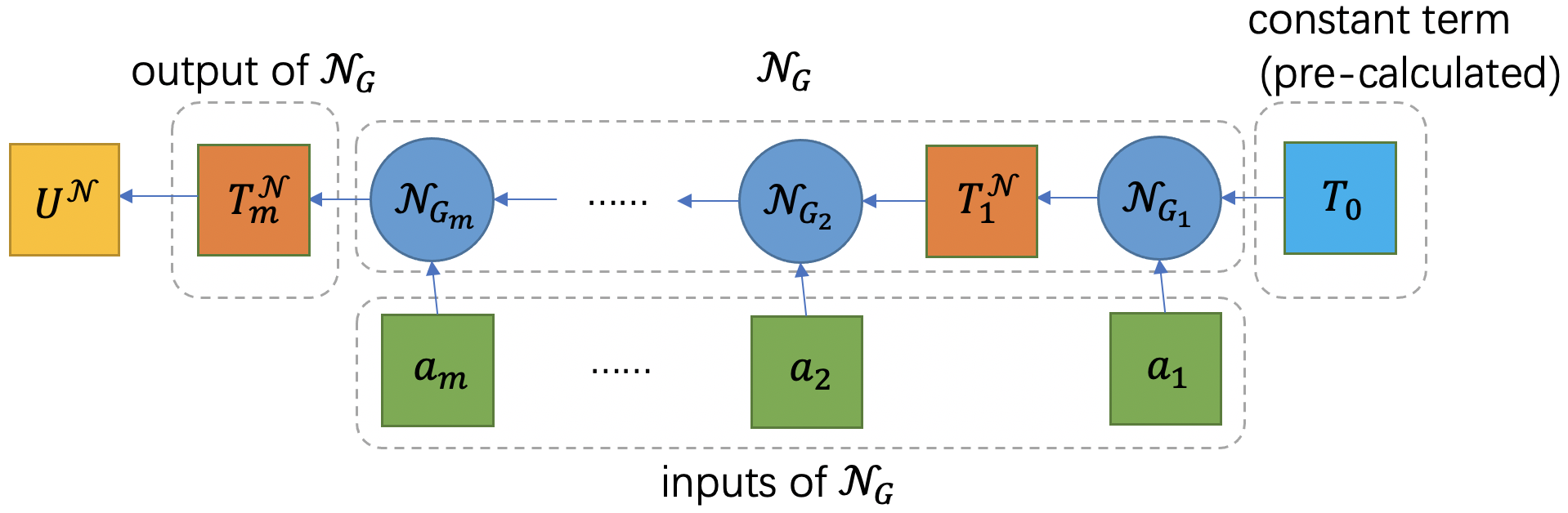} 
   \caption{The structure of $\cN_{G}$ is presented in this figure, where $\mathbf{a}_m$ serves as the input, $T_0 = S^{-1}$ is independent of $a(x,y)$ and can be pre-calculated, and the matrix $T_m^{\cN}$ is the output. Then, $U^{\cN}$ is evaluated by \eqref{defbranch2delliptic}.  }
   \label{ngstructure}
 \end{figure}
The network $\cN_G$ follows the same  compositional structure  in  \eqref{operatorstructure}. Let $T_1^{\cN} = \cN_{G_1}(h^2 a_1, T_0)$ and $T_k^{\cN} = \cN_{G_k}(h^2 a_k, T_{k-1}^{\cN})$, $k=2,\cdots,m$. The next question is how to construct each of $\mathcal{N}_{G_k}$. Even though every $G_k$ has $m^2+1$ inputs and $m^2$ outputs, it is extremely sparse as the following analysis shows that every output depends on 5 inputs only. Let $(T_k)_{ij}$ be the entry in the $i$-th row and $j$-th column of $T_k$. Then
\begin{eqnarray}\label{SMdiscretizeelementwise}
(T_k)_{ij} &=& (T_{k-1})_{ij} - \frac{h^2 a_k}{1+h^2 a_k (T_{k-1})_{kk}} (T_{k-1})_{ik} (T_{k-1})_{kj} \\
&=& \mathcal{R} \Big(h^2 a_k, (T_{k-1})_{ij}, (T_{k-1})_{kk}, (T_{k-1})_{ik}, (T_{k-1})_{kj}\Big),\nonumber
\end{eqnarray}
where the  rational function $\mathcal{R}$ is defined by:
\begin{eqnarray}\label{defrationalelliptic}
y = \mathcal{R}(x_1,x_2,x_3,x_4,x_5):= x_2 - \frac{x_1 x_4 x_5}{1+x_1 x_3}=\frac{x_2 + x_1x_2 x_3 - x_1 x_4 x_5}{1+x_1 x_3}. 
\end{eqnarray}
In the  neural network $\cN_{G_k}$, we can identify  the variables of each functional that calculates $(T_k)_{ij}$, $i,j=1,\cdots, m$ and  approximating the rational function $\mathcal{R}$. By  the facts that $\abs{x_1 x_3} \ll 1$ and that $a(x,y)$ is bounded
and Theorem \ref{thm:rational-apprx-ReLU}, there is a ReLU network $\mathcal{N}_{\mathcal{R}}(x_1,x_2,x_3,x_4,x_5;\theta)$ of size $|\theta| = \mathcal{O}(\ln(\varepsilon^{-1}))$, s.t. 
\begin{eqnarray}\label{networkerror}
\abs{ \mathcal{R}(x_1,x_2,x_3,x_4,x_5)-\mathcal{N}_{\mathcal{R}}(x_1,x_2,x_3,x_4,x_5;\theta)} \leq \varepsilon.
\end{eqnarray}
 
 Let us estimate the number of parameters in the network $\cN_{G}$. 
Suppose that $\cN_{\mathcal{R}}$ has $L_{\mathcal{R}}$ layers and $N_{\mathcal{R}}$ neurons in each layer and $N_{\mathcal{R}}L_{\mathcal{R}} =\mathcal{O}(\ln(\varepsilon^{-1}))$ by Theorem \ref{thm:rational-apprx-ReLU}. To obtain each $(T_k^{\cN})_{ij}$, it  requires $m^2 N_{\mathcal{R}}$ neurons in the first layer of $\cN_{G_k}$ as shown in Figure \ref{ngkstructure}. Then, the   network $\cN_{G}$ has $m L_{\mathcal{R}}$ layers as each $\cN_{G_k}$ has $L_{\mathcal{R}}$ layers.  
  \begin{figure}[htbp] 
    \centering
    \includegraphics[width=5.5in]{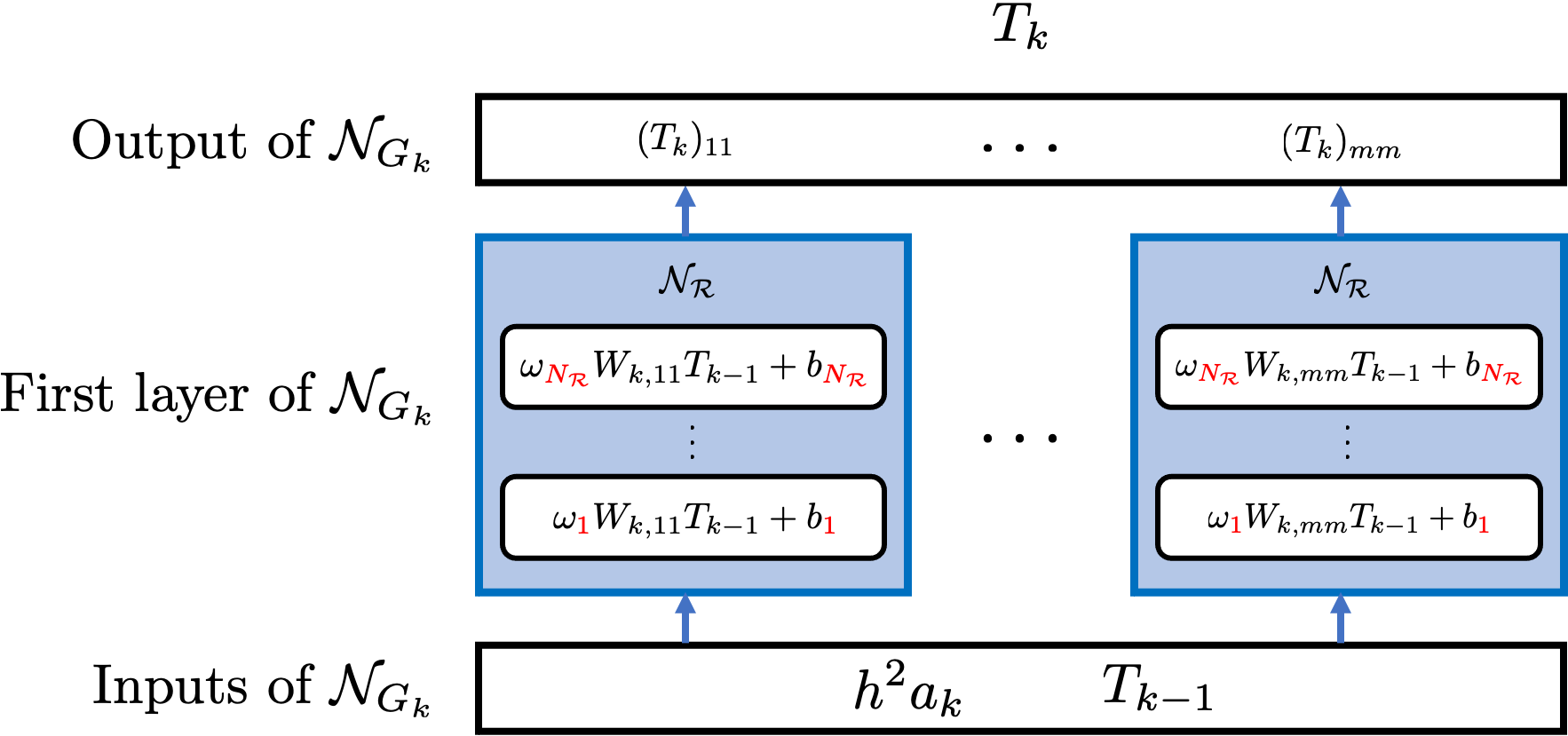}
    \caption{The structure of   $\cN_{G_k}$, $k=1,\cdots,m$,  where $\{\omega_i \}_{i=1}^{N_\mathcal{R}}$ and $\{b_j \}_{j=1}^{N_\mathcal{R}}$ are weights and biases of the first layer of $\cN_{\mathcal{R}}$, respectively, and each $W_{k,ij} \in \mathbb{R}^{5\times m^2}$ is a sparse matrix which has 5 nonzero elements.
    }
    \label{ngkstructure}
 \end{figure} 
 Denote 
\begin{eqnarray}\label{defbranch2delliptic}
  g^{\cN}(\mathbf{a}_m) := T_m^{\cN} F=U^{\cN}.
\end{eqnarray} 

 Then we have the following theorem on the error estimate of the branch network. 

\begin{thm}[Convergence rate of branch networks using  blessed representations]\label{thm:2dreactionbranch}
Let $a(x,y)\in \mathcal{S}$. Let $\cG^f(a)$, $\cG_N^F(\mathbf{a}_m)$ be the analytic solution operator, numerical solution operator to \eqref{2dreactionpde}, respectively.
Let $\{ (x_{j_1}, y_{j_2})\}_{j_1,j_2=1}^{\sqrt{m}}$ be the set of nodes associated with the numerical scheme and $U$ be given by \eqref{defU}. Then there exist a generic constant $C$ and a blessed network $g^{\cN}(\mathbf{a}_m)$ \eqref{defbranch2delliptic} with width $N_{g^{\cN}} = \mathcal{O}(m^2 \ln(m))$ and  depth $L_{g^{\cN}} = \mathcal{O}(m\ln(m))$, such that 
\begin{eqnarray*}
\| U - g^{\cN}(\mathbf{a}_m)\|_{l^{\infty}} \leq C \Big( m^{-\frac{r}{2}}+ m^{-1}\Big),
\end{eqnarray*}
where $r$ is the convergence order of the finite difference scheme used to construct the network.
\end{thm}

Next, the error estimate of DeepONets approximating the solution operator of \eqref{2dreactionpde} can be done using the idea of proof of Theorem \ref{thm:main_burgers} and Remark \ref{rmk:dimdependence}.

\begin{thm}
	For any given $f\in L^{\infty}$, let $\mathcal{G}^f(a)$ be the solution operator and $\cG_{\bN}^f(\mathbf{a}_{m})$ be the DeepONets network.  Then, there exist a branch network $g^{\mathcal{N}}  (\mathbf{a}_m; \Theta)$  of width $N_{g^{\cN}}= \mathcal{O}( m^2 \ln(m))$ and  depth $L_{g^{\cN}} = \mathcal{O}( m\ln(m))$, 
and  ReLU trunk networks $f^{\mathcal{N}}( x ;\theta^{(k)})$ of size $|\theta^{(k)}|=\mathcal{O}(1)$, $k=1,\cdots, p$, such that 
\begin{eqnarray*}
\norm{\cG^f(a)-\cG_{\mathbb{N}}^F (\mathbf{a}^m) }_{L^{\infty}} \leq C\Big( m^{-\frac{r}{2}} + \abs{N_{g^{\cN}} L_{g^{\cN}}}^{-\frac{1}{3}+\epsilon} \Big),
\end{eqnarray*}
where $\epsilon>0$ is arbitrarily small and $C >0$ is independent of $m$, $N_{g^{\cN}}$, $L_{g^{\cN}}$ and $a(x,y)$. Here $ \cG_{\bN}^f(\mathbf{a}_{m}) = g^{\mathcal{N}} (\mathbf{a}_m; \Theta) (f^{\mathcal{N}}( x ;\theta^{(2)}),\ldots, f^{\mathcal{N}}( x ;\theta^{(p)}))^\top$. 
\end{thm}

\begin{proof}
Notice that $a$ and $f \in L^{\infty}$ implies $u \in W^{2,\infty}$. Hence, the interpolation error is $h^2 \sim m^{-1}$. The rest of proof follows the proof of Theorem \ref{thm:main_2dadvreac}.
\end{proof}


\subsection{2D steady advection-diffusion equations}
Consider the 2D steady advection-diffusion equation with a given boundary condition:
\begin{eqnarray}\label{2dadvectionpde}
\left\{\begin{array}{ll}
-\Delta u + \mathbf{a}\cdot  \nabla u = f, \quad \hbox{in} \ \Omega,\\
\mathcal{B}u = 0, \quad \hbox{on} \ \partial \Omega,
\end{array}\right.
\end{eqnarray}
where $ \mathbf{a} = [a_1(x,y), a_2(x,y)]$ and $\mathcal{B}$ can be the Dirichlet, Neumann or Robin boundary operator.   Following the idea in the previous section, we consider a finite difference discretization for \eqref{2dadvectionpde} in order to obtain a linear system. For example,  we apply the central difference scheme to discretize the partial derivatives and denote the corresponding 1-D differential matrix by $D\in \mathbb{R}^{\sqrt{m}\times \sqrt{m}}$, i.e.,  $(D)_{ij} =\pm 1$ when $i=j \mp 1$, and $(D)_{ij} =0$  otherwise, 
 then we have:
\begin{eqnarray}\label{ls2dadvectionpde0}
\Big( S+ D_x + D_y \Big) U_N = F,
\end{eqnarray}
where $S$ is  the stiffness matrix, $D_x = \Lambda_1 (I_{\sqrt{m}\times \sqrt{m}}\otimes D)$ and $D_y = \Lambda_2 (D \otimes I_{\sqrt{m}\times \sqrt{m}})$, where $\otimes$ denotes the Kronecker product for matrices and $\Lambda_k$ are diagonal matrices whose diagonal entries are values of $a_k(x,y)$ at each node, $k=1,2$, respectively. Hence, we can rewrite \eqref{ls2dadvectionpde0} as
\begin{eqnarray}
\Big( S+ h \sum_{i=1}^m a_1^i e_i\hat{R}_i + h \sum_{j=1}^m a_2^j e_j \widetilde{R}_j \Big) U_N = F,
\end{eqnarray}
where $\hat{R}_i$, $\widetilde{R}_i\in \mathbb{R}^{1\times m}$ is the $i$-th row of $ (I_{\sqrt{m}\times \sqrt{m}}\otimes D)$ and $(D \otimes I_{\sqrt{m}\times \sqrt{m}})$, respectively. 
Then, define $T_0 := S^{-1}$ and 
\begin{eqnarray*}
T_k:=
\left\{\begin{array}{ll}
\Big( S+\sum_{i=1}^k h a_1^i  e_i\hat{R}_i\Big)^{-1}, & k=1\cdots,m,\\
\Big( S+ \sum_{i=1}^m ha_1^i e_i\hat{R}_i  + \sum_{j=1}^{k-m} h a_2^j e_j\widetilde{R}_i\Big)^{-1}, & k=m+1,\cdots, 2m,
\end{array}\right.
\end{eqnarray*}
and maps $G_k: \mathbb{R}\times \mathbb{R}^{m\times m} \rightarrow \mathbb{R}^{m\times m}$, s.t.
\begin{eqnarray*}
G_k(\alpha, M):=
\left\{\begin{array}{ll}
M-\frac{\alpha}{1+\alpha(\hat{R}_k M e_k)} M e_k \hat{R}_k M, & k=1,\cdots, m,\\
M-\frac{\alpha}{1+\alpha(\widetilde{R}_{k-m} M e_{k-m})} M e_{k-m} \widetilde{R}_{k-m} M, & k=m+1,\cdots, 2m,
\end{array}\right.
\end{eqnarray*}
which implies that $T_k = G_k(h a_k, T_{k-1})$, where $a_k = a^k_1$ for $k=1,\cdots,m$ and $a_k = a^{k-m}_2$ for $k=m+1,\cdots,2m$. Similar to the discussion of operators for the diffusion-reaction equation,  $G_k$ is a highly sparse map. 
Then, the branch network can be constructed using the blessed representation in the previous section.

\subsection{Concluding remarks}
The analysis for linear advection-diffusion equations also works for   elliptic equations in a divergence form: $-\nabla \cdot (a \nabla u) = f$, where $0<c \leq a(x,y)\leq C$. In fact, one may convert it into a non-divergence form and   $-\Delta u -  {a}^{-1}\nabla a  \cdot \nabla u = a^{-1}f$. Let $\mathbf{b}= -{a}^{-1}\nabla a  $ and $\tilde{f}=a^{-1}f$.  A similar convergence rate can be obtained using  the linearity of the solution operator with respect to the right hand side and the fact that 
$-{a}^{-1}\nabla a$ and $a^{-1}$ can be well approximated by ReLU networks with the input $\mathbf{a}_m$. 

The analysis can be extended to the steady advection-reaction-diffusion equation  $-\Delta u + \mathbf{a} \cdot \nabla u + a_3 u = f$, where $ \mathbf{a}=[a_1,a_2]^\top$ and  $a_1,a_2,a_3,f \in L^{\infty}(\Omega)$. After discretizing the PDE with a finite difference scheme, we may obtain a linear system of  the form: $(S+D_x+D_y + \Lambda) U_N =F$.

For nonlinear problems and time-dependent problems, we expect that additional efficient iterative schemes may be used and thus may lead to extra layers in branch networks. In fact, in each iteration of  the additional iterative schemes, there is a linear system where we can blessed representations.


\section{Proofs}\label{chap5_proofs}
In this section, we present some proofs.

\subsection{Proof of Theorem \ref{thm:deep-o-nets-error}}
\begin{proof}
	By the triangle inequality, 
	\begin{eqnarray} \label{eq:operator-error-split}
	\norm{\cG (u )-\cG_{\bN}(\mathbf{u}_m ) }_{Y} &\leq& \norm{\cG (u)-\cG (\mathcal{I}_mu) }_{Y} +  \norm{\cG (\mathcal{I}_m u )-B_R^\gamma  \cG (\mathcal{I}_m u) }_{Y} \notag \\ 
	&&+ 	 \norm{B_R^\gamma\cG (\mathcal{I}_m u )-\widetilde{\cG}_{\bN} (\mathcal{I}_m u) }_{Y}+\norm{   \widetilde{\cG}_{\bN} (\mathcal{I}_m u)-\cG_{\bN} (\mathcal{I}_m u ) }_{Y},
	\end{eqnarray}
	where $B_R^\gamma v$ is the Bochner-Riesz mean of Fourier series of $v$  
	and 
	\begin{equation*}
	\widetilde{\cG}_{\bN}(u)(y) =\sum_{k=1}^p c_k(\cG (u))f^{\mathcal{N}}(  y;\theta^{(k)}) ,
	\end{equation*}
	given that  $c_k(\cdot)$ is a linear functional defined on $\cG(V)$, $1\leq k\leq p$. By \eqref{eq:operator-holder-continuity}, we have 
	\begin{equation}
	\norm{\cG (u)-\cG (\mathcal{I}_mu) }_{Y}\leq   C\norm{u-\mathcal{I}_m u}_{X}^{\alpha} . 
	\end{equation}
	By Theorem \ref{thm:bochner-riesz-fourier-error} and the fact that $R^d\sim p$, 
	\begin{equation}
	\norm{\cG (\mathcal{I}_m u )-B_R^\gamma  \cG (\mathcal{I}_m u) }_{Y}\leq 
	C \omega_2(\cG (\mathcal{I}_m u ), p^{-1/d}). 
	\end{equation}
	By the assumption that  $u\in V$ and the fact that $c_k(\cdot)$ are linear combinations of Fourier transform of ``$\cdot$", we obtain that  
	$\max_{1\leq k \leq p} \abs{c_k(\cG (\mathcal{I}_m u ))} \leq C$ and thus
	\begin{eqnarray*}\label{thm33part3}
	 \norm{B_R^\gamma\cG (\mathcal{I}_m u )-\widetilde{\cG}_{\bN} (\mathcal{I}_m u) }_{Y}=\norm{\sum_{k=1}^p   c_k(\cG (\mathcal{I}_m u ))(e_k-  f^{\mathcal{N}}(  y;\theta^{(k)}) )}_Y \leq C \sum_{k=1}^p  \norm{e_k-  f^{\mathcal{N}}(  \cdot;\theta^{(k)})}_Y.
	\end{eqnarray*}	
	Recall that $e_k$ is the Fourier basis so we may apply 
	Theorem \ref{thm:conveg-rate-analytic-ReLU} to obtain
	\begin{equation}\label{eq:error-basis-trunk}
	\norm{e_k-  f^{\mathcal{N}}(  \cdot;\theta^{(k)})}_Y \leq  C   \exp\Big(-|\theta^{(k)}|^{\frac{1}{1+d}} \Big) ,
	\end{equation}
	where $|\theta^{(k)}|$ is the size of the  trunk network and 
	$L_{f^{\mathcal{N}}}=C |\theta^{(k)}|^{\frac{1}{1+d}}\ln(|\theta^{(k)}|)$ is the number of layers in the trunk network, which implies that 
	\begin{eqnarray*}
	\norm{B_R^\gamma\cG (\mathcal{I}_m u )-\widetilde{\cG}_{\bN} (\mathcal{I}_m u) }_{Y} \leq C p \exp\Big(-|\theta^{(k)}|^{\frac{1}{1+d}} \Big).
	\end{eqnarray*}
	The last term in \eqref{eq:operator-error-split} can be bounded by 
	\begin{eqnarray*}
	&&\norm{   \widetilde{\cG}_{\bN} (\mathcal{I}_m u)-\cG_{\bN} (\mathcal{I}_m u ) }_{Y} \notag \\
	&\leq & \norm{\sum_{k=1}^p [ c_k(\cG (\mathcal{I}_m u )) - \sum_{i=1}^n c_i^k g^{\mathcal{N}}(\mathbf{u}_m;\Theta^{(k,i)}) ] (f^{\mathcal{N}}(  y;\theta^{(k)}) -e_k) }_{Y} +\norm{\sum_{k=1}^p [ c_k(\cG (\mathcal{I}_m u )) - \sum_{i=1}^n c_i^k g^{\mathcal{N}}(\mathbf{u}_m;\Theta^{(k,i)}) ] e_k  }_{Y}\notag \\
	&\leq & C \sum_{k=1}^p \abs{ c_k(\cG (\mathcal{I}_m u )) - \sum_{i=1}^n c_i^k g^{\mathcal{N}}(\mathbf{u}_m;\Theta^{(k,i)})}.
	\end{eqnarray*}	
		Recall that
	the H\"older continuity \eqref{eq:operator-holder-continuity} and the stability of the interpolation $\norm{\mathcal{I}_m u}_X\leq C \norm{u}_V$, $\forall u\in V$, implied by \eqref{interpltcdt}, we have 
	\begin{equation}\label{eq:linear-func-holder-continuity}
	\sum_{k=1}^p	\abs{c_k(\cG (\mathcal{I}_m u ) ) - c_k(\cG (\mathcal{I}_m v) )}\leq C\norm{\cG (\mathcal{I}_m u)- \cG (\mathcal{I}_m v) }_Y\leq C\norm{  \mathcal{I}_m u-\mathcal{I}_m v}_X^\alpha \leq C \abs{{\bm u}_m -{\bm v}_m}^\alpha,
	\end{equation}
	which implies that 
	$c_k(\cG (\mathcal{I}_m u ))$
	is H\"older continuous   in $(u_1,u_2,\ldots,u_m)$. 
	According to Theorem \ref{thm:complexity-ReLU-generic}, 
	\begin{eqnarray}\label{eq:error-basis-branch} 
	\sum_{k=1}^p  \abs{ c_k(\cG (\mathcal{I}_m u )) - \sum_{i=1}^n c_i^k g^{\mathcal{N}}(\mathbf{u}_m;\Theta^{(k,i)})} \leq Cp \sqrt{m}N_{g^{\cN}}^{-2\alpha/m} L_{g^{\cN}}^{-2\alpha/m},
	\end{eqnarray}
	where $N_{g^{\cN}}$ is the number of neurons per layer of the  branch network and 
	$L_{g^{\cN}}$ is the number of layers of the branch network. 
\end{proof}

\subsection{Proof of Theorem \ref{thm:error-deeponet-linear}}
\begin{proof}
	Recall the conclusion in Theorem \ref{thm:deep-o-nets-error} that 
	\begin{eqnarray*}  
		\norm{\cG (u )-\cG_{\bN}(\mathbf{u}_m ) }_{Y} &\leq& C\norm{u-\mathcal{I}_m u}_{X}^{\alpha}+	C \omega_2(\cG (\mathcal{I}_m u ), p^{-1/d})_q\notag \\
		&&+C \sum_{k=1}^p \abs{ c_k(\cG (\mathcal{I}_m u )) - \sum_{i=1}^n c_i^k g^{\mathcal{N}}(\mathbf{u}_m;\Theta^{(k,i)})} +C \sum_{k=1}^p  \norm{e_k-  f^{\mathcal{N}}(  \cdot;\theta^{(k)})}_Y.
	\end{eqnarray*}
	It's trivial that any linear operator is Lipchitz continuous. On the other hand, if $\mathcal{G}$ is a linear operator, then, by \eqref{eq:deep-o-nets-linear-g}, every $c_k(\mathcal{G}(v)) = \sum_{l=1}^m c^k_l v(x_l)$, $k=1,\cdots, p$, is a linear function with respect to $(v(x_1),v(x_2),\ldots,v(x_m))$. By using the basic fact that any $m$-variable linear function can be exactly expressed by a shallow network that contains $m$ neurons, we can have
	\begin{eqnarray*}
		c_k(\cG (\mathcal{I}_m u )) - \sum_{i=1}^n c_i^k g^{\mathcal{N}}(\mathbf{u}_m;\Theta^{(k,i)}) = 0, \quad \forall k.
	\end{eqnarray*}
Finally, by \eqref{eq:error-basis-trunk}, 
	$C \sum_{k=1}^p  \norm{e_k-  f^{\mathcal{N}}(  \cdot;\theta^{(k)})}_Y \leq C \exp\Big(-\frac{1}{2}|\theta^{(k)}|^{\frac{1}{1+d}}\Big)$, if $|\theta^{(k)}|\geq (2\ln p)^{1+d}$.
\end{proof}

\subsection{Proof of Theorem \ref{thm:oprterror1d}}
\begin{lem}\label{lmbdd}(see \cite{holden2015front})
	Let $u_0\in W^{1,\infty}(-\pi,\pi)$. Then the solution to \eqref{eqtburgers} satisfies the maximum principle
	\begin{eqnarray*}
		\| u(\cdot, t) \|_{L^{\infty}(-\pi,\pi)} \leq \| u_0 \|_{L^{\infty}(-\pi,\pi)}, \quad \forall t>0.
	\end{eqnarray*}
\end{lem}
\begin{proof}
	The proof of Theorem B.1 in \cite{holden2015front}, Appendix B, works for $u_0\in W^{1,\infty}(-\pi,\pi)$.
\end{proof}

\begin{proof}[Proof of Theorem \ref{thm:oprterror1d}]
	First, every $v_i^0$ can be computed accurately and expressed as the product of $V_j$ exactly because $u_0$ is piecewise linear:
	\begin{eqnarray*}
		v_i^0 &=& \exp\Big(-\frac{1}{2\kappa} \int_{0}^{x_i} u_0(s)ds\Big) = \exp\Big( -\frac{1}{2\kappa} \Big[ \sum_{j=1}^i \int_{x_{j-1}}^{x_j} \frac{u_{0,j}-u_{0,j-1}}{h_j}(y-x_{j-1})+u_{0,j-1}dy  \Big] \Big) \\
		&=&  \exp\Big( -\frac{1}{2\kappa} \Big[ \sum_{j=1}^i \frac{1}{2}(u_{0,j}+u_{0,j-1}) h_j\Big] \Big)= \prod_{j=1}^i  \exp(-\frac{u_{0,i}+u_{0,i-1}}{4\kappa}   h_j) = \prod_{j=1}^i V_j.
	\end{eqnarray*}
	Then, notice that 
	\begin{eqnarray*}
		\mathcal{G}(u_0)(\mathbf{x}) - \cG_m  (\mathbf{u}_{0,m}; \mathbf{x}) = u(\mathbf{x})\frac{ \int_{\mathbb{R}}  \mathcal{K}(x,y,t) (\mathcal{I}_m^0v_0 - v_0)(y) dy}{\int_{\mathbb{R}} \mathcal{K}(x,y,t)( \mathcal{I}_m^0 v_0)(y) dy} - \frac{2\kappa \int_{\mathbb{R}}  \partial_x \mathcal{K}(x,y,t) (\mathcal{I}_m^1v_0 - v_0)(y) dy}{\int_{\mathbb{R}} \mathcal{K}(x,y,t)( \mathcal{I}_m^0 v_0)(y) dy}:=I_1+ I_2.
	\end{eqnarray*}
	We observe that 
	\begin{eqnarray*}
		\Big| \partial_x v_0(s)\Big| &=& \Big| v_0(s)  \Big|\cdot \Big| -\frac{1}{2\kappa} u_0(s)\Big|\leq \frac{M_0}{2\kappa} \Big| v_0(s)  \Big| ,\\
		\Big| \partial_{xx} v_0(s)\Big| &=& \Big|\Big(\frac{u_0^2(s)}{4\kappa^2}-\frac{\partial_x u_0(s)}{2\kappa}\Big) \Big| \cdot \Big| v_0(s)\Big| \leq \Big(\frac{M_0^2}{4\kappa^2} + \frac{M_1}{2\kappa}\Big) \Big| v_0(s)\Big|.
	\end{eqnarray*}
	Then, we estimate the two numerators. In every $[x_{i-1}, x_i)$, 
	\begin{eqnarray*}
		\frac{1}{h} \Big| \mathcal{I}_m^0v_0(x) - v_0(x) \Big| \leq \| \partial_x v_0\|_{L^{\infty}[x_{i-1},x_i)} \leq \exp(\frac{M_0h}{2\kappa}) \frac{M_0}{2\kappa} v_0^{i-1} = \Big(\exp(\frac{M_0h}{2\kappa}) \frac{M_0}{2\kappa}\Big) \mathcal{I}_m^0 v_0 .
	\end{eqnarray*}
	So by the comparison principle, 
	\begin{eqnarray*}
		\Big| \int_{\mathbb{R}}  \mathcal{K}(x,y,t) (\mathcal{I}_m^0v_0 - v_0)(y) dy \Big| \leq \Big(\exp(\frac{M_0h}{2\kappa}) \frac{M_0}{2\kappa}h \Big) \int_{\mathbb{R}} \mathcal{K}(x,y,t)( \mathcal{I}_m^0 v_0)(y) dy.
	\end{eqnarray*}
   By Lemma \ref{lmbdd}, $\| u \|_{L^{\infty}}\leq M_0$ and thus  
	\begin{eqnarray*}
		\Big| I_1 \Big| \leq \exp(\frac{M_0h}{2\kappa}) \frac{M_0^2}{2\kappa} h.
	\end{eqnarray*}
Now, we consider the numerator  $	\int_{\mathbb{R}}  \partial_x \mathcal{K}(x,y,t) (\mathcal{I}_m^1v_0 - v_0)(y) dy $. 
	\begin{eqnarray*}
		\int_{\mathbb{R}}  \partial_x \mathcal{K}(x,y,t) (\mathcal{I}_m^1v_0 - v_0)(y) dy = -\int_{\mathbb{R}}  \mathcal{K}(x,y,t) \partial_x(\mathcal{I}_m^1v_0 - v_0)(y) dy.
	\end{eqnarray*}
	In each sub-interval $[x_{i-1}, x_i)$, 
	\begin{eqnarray*}
		\frac{1}{h} \Big| \partial_x \Big(\mathcal{I}_m^1v_0(x) - v_0(x) \Big) \Big| \leq \| \partial_{xx} v_0\|_{L^{\infty}[x_{i-1},x_i)} \leq \exp(\frac{M_0h}{2\kappa}) \Big(\frac{M_0^2}{4\kappa^2} + \frac{M_1}{2\kappa} \Big)v_0^{i-1} = \exp(\frac{M_0h}{2\kappa})\Big(\frac{M_0^2}{4\kappa^2} + \frac{M_1}{2\kappa} \Big) \mathcal{I}_m^0 v_0,
	\end{eqnarray*}
	which implies, by the comparison principle,
	\begin{eqnarray*}
		\Big| I_2 \Big| \leq \exp(\frac{M_0h}{2\kappa})\Big( \frac{M_0^2}{2\kappa} + M_1 \Big) h. 
	\end{eqnarray*}
	For any fixed $\kappa$,   we take $h$ small enough such that $\frac{M_0 h}{2\kappa} \leq \ln 2$, and then 
	\begin{eqnarray*}
		\| \mathcal{G}(u_0) - \cG_m  (\mathbf{u}_{0,m}) \|_{L^{\infty}(-\pi, \pi)\times (0,+\infty)} \leq 2\Big(\frac{M_0^2}{\kappa } + M_1 \Big) h.
	\end{eqnarray*}
\vskip -10pt %
\end{proof}

\subsection{Proof of Theorem \ref{thm:burgersoprtest1}}
\begin{lem}\label{lem:rtnfuncerr}
	Let $M>0$, and let $A$, $B\in \mathbb{R}$, s.t. $|\frac{A}{B}| \leq M$. Let $c_1$, $c_2$ be some given constants. When $2\abs{c_2h }<1$, we have 
	\begin{eqnarray*}
		\Big| \frac{A}{B}-\frac{A+c_1 A h}{B+c_2 B h} \Big|=\abs{\frac{(c_2-c_1)Ah}{B+c_2Bh}} \leq 2M(|c_1|+|c_2|)h.
	\end{eqnarray*}
\end{lem}

\begin{proof} ({Proof of Theorem \ref{thm:burgersoprtest1}})
	For any $\mathbf{x}\in (-\pi,\pi)\times(0,\infty)$, by the triangular inequality and \eqref{defrational}, we have 
	\begin{eqnarray*}
		&& \Big| \mathcal{G}(u_0)(\mathbf{x}) - g^{\mathcal{N}}(\mathbf{u}_{0,m}; \Theta_{\mathbf{x}})\Big|  \\
		&\leq& \Big| \mathcal{G}(u_0)(\mathbf{x}) -  \cG_m  (\mathbf{u}_{0,m}; \mathbf{x})\Big| + \Big| \tilde{\cG}_m  (\mathbf{V}_{0,m}; \mathbf{x})- \tilde{\cG}_m(l^{V}(\mathbf{u}_{0,m});\mathbf{x})\Big| + \Big| \tilde{\cG}_m(l^{V}(\mathbf{u}_{0,m});\mathbf{x}) - g^{\mathcal{N}}(\mathbf{u}_{0,m}; \Theta_{\mathbf{x}}) \Big|\\
		&:=& E_1+ E_2+E_3,
	\end{eqnarray*}
	where $l^V = (l^V_0, l^V_1,\cdots, l^V_{m-1})^\top$, given by
	\begin{eqnarray*}
		l^V_0 = 1, \quad l^V_i = 1 -\frac{u_{0,i}+u_{0,i-1}}{4\kappa}   h_i, \quad i=1,2,\cdots, m-1,
	\end{eqnarray*}
	and $g^{\mathcal{N}}(\mathbf{u}_{0,m}; \Theta_{\mathbf{x}}) = 
	(\tilde{g}^{\mathcal{N}} \circ l^V) (\mathbf{u}_{0,m}) $
	is defined by the composition of a ReLU network and the linear function. Clearly, $l^V_i$ is the linear approximation of $V_i$, $i=0,\cdots, m-1$. So, for any fixed $\kappa$, when $h$ is small enough,
	\begin{eqnarray*}
		|V_i - l^V_i|  \leq \mathcal{O}(h^2),
	\end{eqnarray*}
	which implies, by using $\frac{1}{2}\leq V_i \leq 2$ for small $h$ additionally,
	\begin{eqnarray*}
		\Big| \Big( \prod_{j=1}^i V_j \Big) - \Big( \prod_{j=1}^i l^V_j \Big) \Big| &=& \Big| (V_1-l^V_1)V_2\cdots V_i +\Big[\sum_{k=2}^{i-1} l^V_1 \cdots l^V_{k-1} (V_k - l^V_k) V_{k+1}\cdots V_{i}\Big] +  l^V_1\cdots l^V_{i-1} (V_i-l^V_i) \Big|\\
		&\leq& Ch^2 \Big( V_2\cdots V_i + \Big[\sum_{k=2}^{i} (V_1+Ch^2) \cdots (V_{k-1}+Ch^2) V_{k+1}\cdots V_{i}\Big] \Big) 
		\\
		& \leq& Ch^2 \Big(\sum_{j=0}^i  \frac{v_0^i}{V_j} \Big)+ Ch^2 \leq 2C(2\pi) v_0^i h +  \mathcal{O}(h^2) \leq 5C\pi v_0^i h.
	\end{eqnarray*}
	Therefore, by using Theorem \ref{thm:oprterror1d} and Lemma \ref{lem:rtnfuncerr}, we have $E_1\leq \mathcal{O}(h)$ and $E_2\leq \mathcal{O}(h) $, respectively. On the other hand, since $\tilde{\cG}_m$ is a rational function, so we may apply Theorem \ref{thm:rational-apprx-ReLU} to derive the error estimate for $E_3$, where $r=m$ and $s=m$. Since $\|u_0\|_{L^{\infty}} \leq M_0$, by using Lemma \ref{lmbdd}, the rational function $ |\tilde{\cG}_m  (l^v)| \leq \|u\|_{L^{\infty}} + \mathcal{O}(h) \leq M_0+\mathcal{O}(h) \leq M_0+1$, if $h$ is small enough. Therefore, $ \tilde{\cG}_m  (l^v) = \frac{p}{q} \leq M_0+1 = 2^k$ implies $k = \log_2(M_0+1)$ is a constant.
	Take $\varepsilon = h$ in \eqref{lm12}, and consider $m \sim \frac{2\pi}{h}$, we have that there exists a neural network $\tilde{g}^{\mathcal{N}}(l^v; \Theta_{\mathbf{x}}) $ of size 
	\begin{eqnarray}
	|\Theta_{\mathbf{x}}| = \mathcal{O}\Big(k^7 \ln(\frac{1}{h})^3 +  k m^2\ln(m^2/h) \Big) = \mathcal{O}\Big(\ln(m)^3 +   m^2\ln(m) \Big),
	\end{eqnarray}
	such that 
	\begin{eqnarray}
	\Big| \mathcal{G}(u_0)(\mathbf{x}) -  g^{\mathcal{N}}(\mathbf{u}_{0,m}; \Theta_{\mathbf{x}}) \Big| \leq \mathcal{O}(h)+ h  = \mathcal{O}(h).
	\end{eqnarray}
\end{proof}

\subsection{Proof of Theorem \ref{thm:main_burgers}}
\begin{proof}
To obtain the $L^{\infty}$ error estimate, for any fixed $t$, we take 
\begin{eqnarray}
\cG_{m,p}(u_0) = \sum_{k=1}^p\cG(\mathcal{I}_{m,x}^0u)(y_k) L_i(y):= \mathcal{I}_{p,y}^1 \Big( \cG(\mathcal{I}_{m,x}^0u)\Big),
\end{eqnarray}
where $L_i(x)$ is the piecewise linear nodal basis at $y_i$. By \cite{XuJ2018}, one only needs a network having width $N=3$ and depth $L=1$ to represent a 1-D piecewise linear nodal basis function exactly, i.e.,
\begin{eqnarray*}
L_i(x) = \frac{1}{h_{i-1}} \hbox{ReLU}(x-x_{i-1}) - \Big(\frac{1}{h_{i-1}}+\frac{1}{h_i} \Big) \hbox{ReLU}(x-x_i) +  \frac{1}{h_{i}} \hbox{ReLU}(x-x_{i+1}).
\end{eqnarray*}
And one only needs a network having width $N = \mathcal{O}(k_h)$ and depth $L =\mathcal{O} (\log_2(k_h)+1)$ to approximate a piecewise linear nodal basis function in high dimension exactly, where $k_h$ is the maximum number of neighboring elements in the grid (e.g. $k_h=6$ for 2-D regular triangular mesh). On the other hand, the regularity of the analytical solution is $\cG (u_0 ) \in W^{1,\infty}$. If one uses linear interpolation, then, by Bramble-Hilbert Lemma, there is a generic constant $C$, s.t.
\begin{eqnarray*}
\| \cG(u_0) - \mathcal{I}_p^1 \cG(u_0)\|_{L^{\infty}} \leq C p^{-1} | \cG(u_0) |_{W^{1,\infty}}.
\end{eqnarray*}
Next, by using the maximal principle of linearity of interpolation and the result in Theorem \ref{thm:burgersoprtest1}, we have:
\begin{eqnarray*}
\| \mathcal{I}_p^1 \cG(u_0) - \mathcal{I}_p^1 g^{\mathcal{N}}(\mathbf{u}_{0,m}, \Theta ) \|_{L^{\infty}}  = \| \mathcal{I}_p^1 \Big(\cG(u_0)-g^{\mathcal{N}}(\mathbf{u}_{0,m}; \Theta ) \Big) \|_{L^\infty} \leq \max_{1\leq i\leq p} \abs{ \cG(u_0)(y_i) - g^{\mathcal{N}}(\mathbf{u}_{0,m}; \Theta^{(i)} )} \leq Cm^{-1}.
\end{eqnarray*}
Finally, if the trunk networks $f^{\cN}(y;\theta^{(i)})$ represent $L_i(y)$, $i=1,\cdots, p$ exactly, then we have
\begin{eqnarray*}
\norm{\cG (u_0 )-\cG_{\bN}(\mathbf{u}_m ) }_{L^\infty} &\leq& \| \cG(u_0) - \mathcal{I}_p^1 \cG(u_0)\|_{L^{\infty}} + \| \mathcal{I}_{p}^1 \cG(u_0) - \mathcal{I}_{p}^1 g^{\mathcal{N}}(\mathbf{u}_{0,m}; \Theta ) \|_{L^\infty} + \| \mathcal{I}_{p}^1 g^{\mathcal{N}}(\mathbf{u}_{0,m}; \Theta ) - \cG_{\bN}(\mathbf{u}_m )\|_{L^\infty}\\
&\leq& C\Big( p^{-1} + m^{-1}  \Big),
\end{eqnarray*}
where $m^{-1} \sim |\Theta|^{-\frac{1}{2}+\epsilon}$ for any arbitrary small $\epsilon>0$.
\end{proof}

\subsection{Proofs in the Section 4.5}
\begin{proof}[Proof of Lemma \ref{prop:1dadvecrationalerr}]
	First, we consider the following error. For $x\in (x_{k-1}, x_k]$, 
	\begin{eqnarray*} 
		\Big| \Big( A^{\pm1}g - \mathcal{I}^0(A^{\pm1}g)\Big)(x) \Big|&=&\Big| \int_{x_{i-1}}^x \Big( A_x(s)g(s) + A(s)g_x(s) \Big)ds \Big|= \Big| \int_{x_{i-1}}^x A(s)\Big( \mp g(s)a(s) + g_x(s) \Big)ds\Big|\\
		&\leq & \|A\|_{L^{\infty}(x_{i-1},x_i)} \Big(\|g\|_{L^{\infty}(x_{i-1},x_i)} \|a\|_{L^{\infty}(x_{i-1},x_i)}+\|g_x\|_{L^{\infty}(x_{i-1},x_i)}\Big) (x - x_{i-1})\\
		&\leq & \mathcal{I}^0(A^{\pm1}) \exp(M_1h_i) \Big(\|g\|_{L^{\infty}(x_{i-1},x_i)} \|a\|_{L^{\infty}(x_{i-1},x_i)}+\|g_x\|_{L^{\infty}(x_{i-1},x_i)}\Big) h_i,
	\end{eqnarray*}
	where $h_i = x_i - x_{i-1}$, $h=\max_{i} h_i$. In particular, if $g$ is constant or piecewise constant, then
	\begin{eqnarray*}
		\Big| \Big( A^{\pm1}g - \mathcal{I}^0(A^{\pm1}g)\Big)(x) \Big| \leq \Big| \mathcal{I}^0(A^{\pm1}g) \Big|  M_1 \exp(M_1h_i) h_i,
	\end{eqnarray*}
	which implies that 
	\begin{eqnarray*}
		\frac{\mathcal{A}_- (\mathbbm{1})(x)}{\mathcal{A}_- (\mathbbm{1})(L)} - \frac{\mathcal{A}_-^N (\mathbbm{1})(x)}{\mathcal{A}_-^N (\mathbbm{1})(L)} \leq 2M_1 \exp(M_1 h) h,
	\end{eqnarray*}
	and there exists a uniform constant C, such that for any $x \in [0,L]$,
	\begin{eqnarray*}
	\abs{\Big( \mathcal{A}_-^N \circ  \mathcal{A}_+^N \Big)(f)(x) - \Big( \mathcal{A}_- \circ  \mathcal{A}_+ \Big)(f)(x)} \leq Ch.
	\end{eqnarray*}
	Then the proof can be completed by the triangular inequality and the two error estimates above.
\end{proof}

\begin{proof}[Proof of Theorem \ref{thm:1dadvecnetworkerr}]
Basically, we follow the idea of proof of Theorem \ref{thm:burgersoprtest1}. The key step is to figure out the degree ($r$), number of variables ($d$) and number of terms ($s$) of $\cG^f_m$'s numerator and denominator, respectively, in order to apply Theorem \ref{thm:rational-apprx-ReLU} to determine the size of network $g^{\cN}_f$. Define $V_m:=[v_1,\cdots, v_m]$, where 
\begin{eqnarray*}
v_0 =1, \quad v_i = \exp(a_i h_i), \quad \hbox{given that} \ h_i = x_i - x_{i-1}, \quad i=1,\cdots,m,
\end{eqnarray*}
whence $A(x_j) = \prod_{k=0}^j \frac{1}{v_k}$ and $A^{-1}(x_j) = \prod_{k=0}^j v_k$. Notice that, for any $J=1,\cdots,m$,
\begin{eqnarray*}
\mathcal{A}^N_- (\mathbbm{1})(x_J) = \sum_{j=1}^J A^{-1}(x_j) f_j h_j = \sum_{j=1}^J h_j f_j \Big( \prod_{k=0}^j v_k\Big),
\end{eqnarray*}
which is a $J$-th degree $J$-variable ($\{v_1,\cdots, v_m\}$) polynomial with $J$ terms, and 
\begin{eqnarray*}
\Big( \mathcal{A}_-^N \circ  \mathcal{A}_+^N \Big)(f)(x_J) = \sum_{i=1}^J \frac{h_i}{A(x_i)} \cA^N_+(f)(x_i) = \sum_{i=1}^J \frac{h_i}{A(x_i)} \Big(\sum_{j=1}^i A(x_j) f_j h_j \Big) = \sum_{i=1}^J \sum_{j=1}^i h_i h_j f_j \frac{A(x_j)}{A(x_i)},
\end{eqnarray*}
where $\frac{A(x_j)}{A(x_i)}=\prod_{k=j+1}^i v_k$ for $0\leq j< i \leq m$, and $\frac{A(x_j)}{A(x_i)} = 1$ for $0\leq i=j\leq m$, which implies that $\Big( \mathcal{A}_-^N \circ  \mathcal{A}_+^N \Big)(f)(x_J)$ is a $J$-th degree $J$-variable ($\{v_1,\cdots, v_m\}$) polynomial with $J(J+1)/2$ terms. If, for any $J=1,\cdots,m$, we rewrite \eqref{ellipticrationaldef} as
\begin{eqnarray*}
\cG^f_m(\mathbf{a}_m;x) = \tilde{\cG}^f_m(V_m;x):=\frac{\Big( \mathcal{A}^N_- (\mathbbm{1})(x)\Big) \Big( ( \mathcal{A}_-^N \circ  \mathcal{A}_+^N )(f)(x_m)\Big) - \Big( \mathcal{A}^N_- (\mathbbm{1})(x_m )\Big) \Big( ( \mathcal{A}_-^N \circ  \mathcal{A}_+^N \Big)(f)(x)\Big)}{\mathcal{A}^N_- (\mathbbm{1})(x_m)}.
\end{eqnarray*}
The numerator of $\tilde{\cG}^f_m(V_m;x)$ is a $2m$-th degree $m$-variable polynomial with at most $\mathcal{O}(m^3)$ terms. The rest of this proof is similar to the proof of Theorem \ref{thm:burgersoprtest1}.
\end{proof}

 
 \subsection{Proof of Theorem \ref{thm:2dreactionbranch}}
 \begin{proof}
Let $A \in \mathbb{R}^{m\times m}$, $b = [b_1,\cdots, b_m]^\top\in \mathbb{R}^m$. Denote $\| A\| =\max_{ij} |(A)_{ij}|$, $\| b\|_{\infty} = \max_{i} |b_i|$. 
Herein, we consider the error estimate $\| U_N - U^{\cN}\|_{\infty}$. First, we consider the error in each layer of the network:
\begin{eqnarray*}
\|T_k - T_k^{\cN}\| \leq \| G_k (h^2 a_k, T_{k-1}) - G_k(h^2 a_k, T_{k-1}^{\cN})\| + \|G_k(h^2 a_k, T_{k-1}^{\cN}) - \cN_{G_k}(h^2 a_k, T_{k-1}^{\cN})\|:= E_1^k+E_2^k,
\end{eqnarray*}
for any $k=1,2\cdots,n$. First, $E_1^k$ is estimated by
\begin{eqnarray*}
|E_1^k| &=&\abs{ \mathcal{R} \Big(h^2 a_k, (T_{k-1})_{ij}, (T_{k-1})_{kk}, (T_{k-1})_{ik}, (T_{k-1})_{kj}\Big) - \mathcal{R} \Big(h^2 a_k, (T_{k-1}^{\cN})_{ij}, (T_{k-1}^{\cN})_{kk}, (T_{k-1}^{\cN})_{ik}, (T_{k-1}^{\cN})_{kj}\Big)}\\
&\leq& \abs{ (T_{k-1})_{ij}-(T_{k-1}^{\cN})_{ij} } + \abs{\frac{h^2 a_k}{1+h^2 a_k (T_{k-1})_{kk}} (T_{k-1})_{ik} (T_{k-1})_{kj} - \frac{h^2 a_k}{1+h^2 a_k (T_{k-1}^{\cN})_{kk}} (T_{k-1}^{\cN})_{ik} (T_{k-1})_{kj}^{\cN}} \\
&\leq&  \Big(1+ \frac{\partial \mathcal{R}}{\partial x_3}\Big|_{(T_{k-1})_{kk}} + \frac{\partial \mathcal{R}}{\partial x_4}\Big|_{(T_{k-1})_{ik}} + \frac{\partial \mathcal{R}}{\partial x_5}\Big|_{(T_{k-1})_{kj}}
\Big) \| T_{k-1} - T_{k-1}^{\cN}\| \\
&\leq& \Big( 1+ \frac{h^4 a_k^2 (T_{k-1})_{ik} (T_{k-1})_{kj}}{(1+h^2 a_k (T_{k-1})_{kk})^2} + \frac{h^2 a_k [(T_{k-1})_{ik}+ (T_{k-1})_{kj}]}{(1+h^2 a_k (T_{k-1})_{kk})} 
\Big)\| T_{k-1} - T_{k-1}^{\cN}\|.
\end{eqnarray*}
Notice that there exists a uniform constant $C_1>0$ s.t. $\forall k$, $\| T_{k}\| \leq C_1$, because every $T_k$ is the inverse of $(S+\Lambda_k)$ for $a_k(x,y)$ that vanishes at all of $k$-th to $m$-th nodes. Hence, suppose $h<1$, then there exists a uniform constant $c =1+  h^4 C_0 C_1^2 + 2h^2 C_0 C_1 \leq 1+  C_0 C_1^2 + 2 C_0 C_1$, s.t.
$E_1^k \leq c \| T_{k-1} - T_{k-1}^{\cN} \|$. And $E_2^k$ is the network error given by \eqref{networkerror}. Notice that $E_1^1 = 0$. If we set $E_2^k = \varepsilon$ for all $k$, we have 
\begin{eqnarray*}
\| T_m - T_m^{\cN}\| \leq cm \varepsilon,
\end{eqnarray*}
which implies that
\begin{eqnarray*}
\| U_N - U^{\cN}\|_{l^\infty} = \| (T_m - T_m^{\cN})F \|_{l^\infty} \leq \Big( cm^2 \| F\|_{\infty}\Big)  \varepsilon = c\| F\|_{\infty}\ m^{-1},
\end{eqnarray*}
where we take $\varepsilon \sim m^{-3}$. The proof may be completed by the triangular inequality.
 
 \end{proof}
 
 
 \subsection{Proof of Theorem \ref{thm:main_2dadvreac}}
 \begin{proof}
 Notice that the branch network only evaluates at the nodes given by the finite difference scheme, so $p=m$. On the other hand, $a_1$, $a_2$, $a_3$, $f\in L^{\infty}$ implies $u \in W^{2,2}$. Then, by Sobolev embedding theorem that $W^{1,2}(\Omega) \hookrightarrow L^{\alpha}(\Omega)$ for an arbitrary large number $\alpha$ when $d=2$, we can have
 \begin{eqnarray*}
 |u |_{W^{2,\alpha}} \leq \| \Delta u \|_{L^{\alpha}} \leq (\| a_1 \|_{L^{\infty}} + \| a_2 \|_{L^{\infty}}) \| \nabla u\|_{L^{\alpha} } + \|a_3 \|_{L^{\infty}} \| u \|_{L^\infty} + \| f\|_{L^{\infty}},
 \end{eqnarray*}
 which implies, for some small $\epsilon$ depending on $\alpha$,
 \begin{eqnarray*}
 \| u - \mathcal{I}^1_m u\|_{L^{\infty}} \leq C h^{2-\epsilon} | u |_{W^{2,\alpha}}.
 \end{eqnarray*}
Hence, due to the error of the trunk networks is 0, we have
\begin{eqnarray*}
\norm{\cG^f(a_1,a_2,a_3)-\cG_{\mathbb{N}}^F (\mathbf{a}_1^m, \mathbf{a}_2^m, \mathbf{a}_3^m) }_{L^{\infty}} &\leq& \| u  - \mathcal{I}_m^1 u \|_{L^{\infty}} + \| \mathcal{I}_{p}^1 u - \mathcal{I}_{p}^1 g^{\mathcal{N}} \|_{L^\infty} + \| \mathcal{I}_{m}^1 g^{\mathcal{N}} - \cG_{\bN}(\mathbf{u}_m )\|_{L^\infty}\\
&\leq& C\Big(m^{-1+\epsilon} + m^{-\frac{r}{2}} + \abs{N_{g^{\cN}} L_{g^{\cN}}}^{-\frac{1}{3}+\epsilon} \Big),
\end{eqnarray*}
 where $h^2 \sim m^{-1}$ and $\epsilon$ is arbitrary small.
 
 \end{proof}



\section*{Acknowledgment}
This material is based upon work supported by the Air Force Office of Scientific Research under award number FA9550-20-1-0056.
This work is also supported by the DOE PhILMs project (No. DE-SC0019453) and the DARPA-CompMods grant HR00112090062.

\appendix
\section{Approximation using Fourier expansion}

\begin{thm}\label{thm:bochner-riesz-fourier-error}
	Let $f\in L^q([-\pi, \pi)^d) $,  $1\leq q \leq \infty$. 
	For $d\geq 1$ and $\gamma>\gamma* =(d-1)\abs{\frac{1}{q}-\frac{1}{2}}$, 
	\begin{equation}
	\norm{ B_R^\gamma f -f}_{L^q} \leq C_{d,\gamma,q} \,\omega_2(f;R^{-1})_q.
	\end{equation}
	Here $C_{d,\gamma,q}$ is a positive  constant independent of $R$ and $f$.
\end{thm}
\begin{proof}
	When $d>1$, this is exactly Theorem 3.4.5 of \cite{LuYan-B13}. 
	According to the proof of  Theorem 3.4.5 in  \cite{LuYan-B13},  the conclusion is  also valid 
	when $d=1$ as long as $\gamma>0$. 
\end{proof}

\section{Convergence rates of deep ReLU networks for continuous functions}
\begin{thm}[\cite{ShenYZ2019}]
	\label{thm:complexity-ReLU-generic}
	Given $f\in C([0,1]^d)$, for any integers $L,N>0$, and $p\in [1,\infty]$, 
	there exists   a ReLU feedforward network  $f^{\mathcal{N}}$ with width $C_1\max\big\{d\lfloor N^{1/d}\rfloor,\, N+1\big\}$
	and depth $12L+C_2$
	such that 
	\begin{equation*}
	\|f-f^{\mathcal{N}}\|_{L^p([0,1]^d)}\le 19\sqrt{d}\,\omega(f,N^{-2/d}L^{-2/d})_\infty,
	\end{equation*}
	where $\omega(f,\cdot)$ is the modulus of continuity of $f$ defined via
	\begin{equation*}
	\omega(f,\delta)_\infty= \sup_{\abs{ z}\leq r}\norm{{f({\cdot+z})-f({\cdot})}}_{L^\infty}, {x},\,{y}\in [0,1]^d , 
	\end{equation*}  and $C_1=12$ and $C_2=14$ if $p\in [1,\infty)$;  $C_1=3^{d+3}$ and $C_2=14+2d$ if $p=\infty$. 
\end{thm}

\begin{thm}[\cite{yarotsky2018optimal}]  For any continuous function
	Let $f:[0.1]^n\to \Real$ be with modulus of continuity $\omega_f$, there is a deep ReLU
	network $\tilde{f}$ with depth $L\leq c_0(n)$ and size $W$ such that 
	\begin{equation*}
	\norm{f-\tilde{f}}_{L^\infty}\leq c_1(n)\omega_f(c_2(n)W^{-2/n}), \quad \text{ where } c_0(n),\, c_1(n),\, c_2(n)>0. 
	\end{equation*}
\end{thm}


\begin{thm}[Theorem 3 in \cite{Sch20}]Let $p \in[1, \infty) $ and assume that there exists  a constant $C>0$ such that $\abs{f(\mathbf{x})- f(\mathbf{y})}\leq C\norm{ \mathbf{x}- \mathbf{y}}_{l^\infty} ^{\beta}$ with  $\beta \in[0,1]$,  for all $\mathbf{x}, \mathbf{y} \in[0,1]^{d}$.
	Then, there exists a deep  ReLU network $ \tilde{f}$ with
	$ K+3$ hidden layers, network architecture $\left(2 K+3,\left(d, 4 d, \ldots, 4 d, d, 1,2^{K d}+1,1\right)\right)$ and all network weights bounded in absolute value by $2\left(K \vee\|f\|_{\infty}\right) 2^{K(d \vee(p \beta))},$ such that
	$$
	\|f-\tilde{f}\|_{p} \leq\left(16 C+2\|f\|_{L^\infty}\right) 2^{-\beta K}.
	$$
	where $\tilde{f}(x_1,\ldots, xd)  = \tilde{g}(\sum_{k=1}^d3^{-k}\widetilde{\phi}_K(x_k))$ and 
	$\tilde{g} $ is a ReLU network with network architecture $\left(2,\left(  1,2^{K d}+1\right)\right)$
	and $\tilde{\phi}_K$ is a ReLU network with network architecture   $\left(2 K,\left(d, 4 d, \ldots, 4 d, d\right)\right)$.
	
\end{thm} Here $C$ can depend on $d$, e.g. $\abs{f(x)-f(y)}\leq 
C_1\norm{x-y}_{l^2}^\beta \leq C_1d^\beta \norm{x-y}_{l^\infty}^\beta$. 
The theorem leads to a 
rate  of $2^{-K\beta}$ using of the order of $2^{Kd}$ network parameters-one obtains the order $(2^{Kd})^{-\frac{\beta}{d}}$. 
%

The following is from  \cite{OpsSZ2019} on convergence rate of ReLU networks for analytic functions.
\begin{thm}[Theorem 3.7 in \cite{OpsSZ2019}]\label{thm:conveg-rate-analytic-ReLU}
	Fix $d\in \mathbb{N}$. Assume that $u:[-1,1]^d \rightarrow \mathbb{R}$ can be analytically extended to Bernstein Ellipse $\mathcal{E}_{\mathbf{\rho}}$ with radius $\mathbf{\rho}=(\rho_j)_{j=1}^d$. Then, there exist constants $\beta' = \beta'(\mathbb{\rho},d)>0$ and $C = C(u,\mathbf{\rho},d)>0$, and for every $\mathcal{N}\in \mathbb{N}$ there exists a ReLU neural network $\tilde{u}_{\mathcal{N}}:[-1,1]^d\rightarrow \mathbb{R}$ satisfying
	\begin{eqnarray*}
		\hbox{size}(\tilde{u}_{\mathcal{N}}) \leq \mathcal{N}, \quad \hbox{depth}(\tilde{u}_{\mathcal{N}}) \leq C\mathcal{N}^{\frac{1}{d+1}} \log_2(\mathcal{N}),
	\end{eqnarray*}
	and the error bound
	\begin{eqnarray*}
		\| u(\cdot) - \tilde{u}_{\mathcal{N}}(\cdot)\|_{W^{1,\infty}([-1,1]^d)} \leq C \exp(-\beta' \mathcal{N}^{\frac{1}{1+d}}).
	\end{eqnarray*}
\end{thm}

  

\end{document}